\documentclass{amsart}
\usepackage{macros}
\usepackage{bbm}
\makeatletter
\newcommand{\addresseshere}{%
  \enddoc@text\let\enddoc@text\relax
}
\makeatother

\upsymbols{Tor,Isom,Div,Hom,Br,inv,WC,nr,red,GL,SL,Aut,tors,Hom,Inv,Prob}
\upsymbols{Gal,res,im,coker,Frob,rk,Res,Ext,disc,Id,FS,cyc,PGL,cor,Vol,loc}
\bbsymbols{A,G,Z,Q,R,F,C,E,P}
\scrsymbols{F,C,S,L,M}
\calsymbols{O,T,F,G,C,D,A,I}
\fraksymbols{p,P}
\usepackage[left=4cm,right=4cm]{geometry}

\newcommand{\infl}{\textup{inf}}

\newcommand{\sel}[1]{\textup{Sel}_{#1}}
\newcommand{\cores}{\cor}

\title{Quadratic Twists as Random Variables}
\author{Ross Paterson}
\makeatletter
\hypersetup{
pdftitle={\@title},
pdfauthor={Ross Paterson},
pdfsubject={Number Theory},
pdfkeywords={}
}
\makeatother
\address{School of Mathematics, University of Bristol, Bristol, BS8 1TW, UK, and the Heilbronn Institute for Mathematical Research, Bristol, UK.}
\email{rosspatersonmath@gmail.com}
\urladdr{\url{https://ramifiedprime.github.io}}

\begin{document}

\begin{abstract}
    Let $D\neq 1$ be a fixed squarefree integer.  For elliptic curves $E/\QQ$, writing $E_D$ for the quadratic twist by $D$, we consider the question of how often $E(\QQ)$ and $E_{D}(\QQ)$ generate $E(\QQ(\sqrt{D}))$.  We bound the proportion of $E$, ordered by height, for which this is not the case, showing that it is very small for typical $D$.

    The central theorem is concerned with intersections of $2$-Selmer groups of quadratic twists.  We establish their average size in terms of a product of local densities.  We additionally propose a heuristic model for these intersections, which explains our result and similar results in the literature.  This heuristic predicts further results in other families.
\end{abstract}

\maketitle
\tableofcontents
\section{Introduction}
Let $K=\QQ(\sqrt{D})$ be a quadratic number field, with Galois group $G$, and let $E:y^2=x^3+Ax+B$ be an elliptic curve defined over $\QQ$.  The quadratic twist of $E$ by $D$ is given by
\[E_D:Dy^2=x^3+Ax+B.\]
There is an obvious isomorphism $\varphi_D:E_D\to E$ defined over $K$, namely the one given by $\left(x,y\right)\mapsto \left(x,y\sqrt{D}\right)$.
This morphism identifies $E_D(\QQ)$ with the set of points of $E(K)$ on which the generator of $G$ acts as multiplication by $-1$.

\subsection{Mordell--Weil Groups}
We are motivated by the following question.

\begin{ques}\label{ques:MW}
    Is the natural map $\eta_{E,D}:E_D(\QQ)\times E(\QQ)\to E(K)$, defined by
    \[(P, Q)\mapsto \varphi_{D}(P)+Q,\]
    an isomorphism?
\end{ques}

It is easy to see that the elements of $E(\QQ)[2]$ present an obstruction to $\eta_{E,D}$ being injective.  However, even when $E(\QQ)[2]=0$ there is no uniform answer to \Cref{ques:MW}.  The examples below exhibit each answer.

\begin{example}[Obtained using \texttt{MAGMA} \cite{MR1484478}]
Consider the field $K=\QQ(i)$ where $i$ is a primitive $4$th root of unity, and the elliptic curves
\begin{align*}
E_1&:y^2=x^3+x+1,\\
E_2&:y^2=x^3-3x+1.
\end{align*}
These curves both have no torsion points over $K$ and rank $2$, with generators as follows:
\begin{align*}
E_1(K)&=\gp{(i,1), (-i,1)}\cong \ZZ^2,\\
E_2(K)&=\gp{(1,i), (0,1)}\cong \ZZ^2.
\end{align*}
On one hand, $E_1(\QQ)=\gp{(0,-1)}$ and $E_1^{-1}(\QQ)=\gp{(-1,1)}$, so the image of $\eta_{E_1,-1}$ is spanned by $(0,-1)=(-i,1)+(i,1)$ and $(-1,i)=(-i,1)-(i,1)$.  Thus $\im(\eta_{E_1,-1})\subsetneq E_1(K)$ is an index $2$ subgroup, and so the answer to \Cref{ques:MW} is no for $E_1$.

On the other hand, clearly $(1,1)\in E_2^{-1}(\QQ)$, $(1,i)=\eta_{E_2,-1}(1,1)$ and $(0,1)\in E_2(\QQ)$.  In particular, the answer to \Cref{ques:MW} for $E_2$ is yes.
\end{example}

This article will consider \textit{how often} the answer to \Cref{ques:MW} is negative, i.e. how often $\eta_{E,D}$ fails to be an isomorphism as $E/\QQ$ varies in the family of all elliptic curves (ordered by height).  Our main result, \Cref{thm:INTROTHM imprecise on decomp}, shows that for a typical fixed $D$ and a randomly selected elliptic curve $E/\QQ$, it is very likely that $\eta_{E,D}$ is an isomorphism.

More precisely: throughout the article we let
\begin{equation}\label{eq:epsilon}
\Epsilon=\set{E_{A,B}:y^2=x^3+Ax+B: \substack{A,B\in\ZZ,\\\gcd(A^3,B^2)\textnormal{ is }12^{th}\textnormal{-power free,}\\\textnormal{and }4A^3+27B^2\neq 0}}.
\end{equation}
It is well known, see e.g. \cite{silverman2009arithmetic}*{III.1}, that every elliptic curve defined over the rational numbers has precisely one model in the set $\Epsilon$.  The (na\"ive) height of $E=E_{A,B}\in\Epsilon$ is defined to be $H(E)=\max\{4\abs{A}^3,27B^2\}$, and for every positive real number $X$, we write $\Epsilon(X)$ for the finite subset of $\Epsilon$ of curves which have height at most $X$.

\begin{theorem}[\Cref{thm:THM imprecise on decomp}]\label{thm:INTROTHM imprecise on decomp}
    For each fixed squarefree integer $D$, the probability that $\eta_{E,D}$ is not an isomorphism is small.  More precisely,
    \[\limsup_{X\to\infty}\frac{\#\set{E\in\Epsilon(X)~:~\eta_{E,D}\textnormal{ is not an isomorphism}}}{\#\Epsilon(X)}\ll \braces{\frac{23}{24}}^{\omega(D)},\]
    where the implied constant is independent of $D$. 
\end{theorem}
\begin{rem}
The analogous question modulo torsion has already been studied in quadratic twist families of elliptic curves with full $2$-torsion by Morgan--Paterson \cite{MR4400944}*{Theorem 1.7}.  There it is shown that, modulo $2$-torsion, $\eta_{E,D}$ is an isomorphism for $100\%$ of quadratic twists.
\end{rem}
  \Cref{thm:INTROTHM imprecise on decomp} arises as a consequence of our investigation of a seemingly different question, on the behaviour of $2$-Selmer groups of quadratic twists of elliptic curves as random variables, which we now move on to discuss.

\subsection{Intersections of Selmer Groups}
For each squarefree integer $D$, the $\QQ(\sqrt{D})$-isomorphism $\varphi_D:E_D\to E$ restricts to an isomorphism (over $\QQ$) of Galois modules 
\[E_D[2]\cong_{G_\QQ} E[2].\]
In particular, $\varphi_D$ induces an isomorphism $H^1(\QQ, E_D[2])\cong H^1(\QQ, E[2])$ and so under this identification we can view the $2$-Selmer groups of $E$ and $E_D$ as embedded in the same space:
\[\sel{2}(E/\QQ),\sel{2}(E_D/\QQ)\subseteq H^1(\QQ, E[2]).\]
As $E$ varies, one thinks of the Selmer groups as appropriately random subspaces of $H^1(\QQ, E[2])$ (see \cite{Poonen_2012}).  It is natural to ask how often they encounter one another.

\begin{ques}\label{ques:INTRO Selmer intersection how often}
    For a fixed $D$, and $E/\QQ$ varying in the family of all elliptic curves, what is the average size of $\sel{2}(E/\QQ)\cap\sel{2}(E_D/\QQ)$?
\end{ques}

\noindent In this article we present both an answer to this question, and a heuristic to explain it.  The heuristic explains similar results for other families in the literature, and also predicts the behaviour for further families.
\subsubsection{Theorems}
In this article we obtain the average size of this intersection as a product of local densities.

\begin{theorem}[\Cref{thm:average size of intersection is product of densities}]\label{thm:INTRO average size of intersection is product of densities}
    Let $S$ be a finite set of squarefree integers.  Then,
    \[\lim_{X\to\infty}\frac{\sum_{E\in\Epsilon(X)}\#\bigcap\limits_{D\in S}\sel{2}(E_D/\QQ)}{\#\Epsilon(X)}=1+2\prod_{v\in\places_\QQ}\delta_{S,v}>1,\]
    where the $\delta_{S,v}$ are local densities defined in \Cref{thm:average size of intersection is product of densities} depending on $S$, and $\Omega_\QQ$ denotes the set of places of $\QQ$.
\end{theorem}

This then enables us to show that our Selmer groups meet nontrivially a positive proportion of the time.
\begin{corollary}[\Cref{thm:prob nontrivial is positive}]\label{thm:INTRO prob nontrivial is positive} Let $D$ be a squarefree integer.  Then,
    \[\liminf_{X\to\infty}\frac{\#\set{E\in\Epsilon(X)~:~\sel{2}(E/\QQ)\cap \sel{2}(E_D/\QQ)\neq 0}}{\#\Epsilon(X)}>0.\]
\end{corollary}

\noindent However, explicitly determining the local densities $\delta_{S,v}$ we show that this proportion is typically quite small.

\begin{theorem}[\Cref{thm:prob nonzero goes to zero}]\label{thm:INTRO prob nonzero goes to zero}
    Let $D$ be a squarefree integer, then
    \[\limsup_{X\to\infty}\frac{\#\set{E\in\Epsilon(X)~:~\sel{2}(E/\QQ)\cap\sel{2}(E_D/\QQ)\neq 0}}{\#\Epsilon(X)}\ll \braces{\frac{23}{24}}^{\omega(D)},\]
    where the implied constant is independent of $D$.
\end{theorem}
\subsubsection{Heuristic}
The result in \Cref{thm:INTRO prob nontrivial is positive} differs from an analogous result for certain quadratic twist families in \cite{MR4400944}*{Theorem 6.1 + Remark 6.3}, where it is shown that there the average size of $\sel{2}(E/\QQ)\cap \sel{2}(E_D/\QQ)$ is exactly $1$.  In this article we also pose a heuristic which both explains these results and makes further predictions for other families of elliptic curves.

We now give a short summary of our heuristic, in lieu of details which can be found in \Cref{sec:Heuristic}.  A standard heuristic model for the behaviour of $\sel{2}(E/\QQ)$ is to model it as an intersection of $2$ random maximal isotropic subspaces in an appropriate quadratic space, as first suggested by \cite{Poonen_2012}.  Following from this, one can naturally interpret $\sel{2}(E/\QQ)\cap\sel{2}(E_D/\QQ)$ as an intersection of $3$ such spaces, and so na\"ively one might model it as though these $3$ spaces are random.  We show that this na\"ive model suggests that the intersection should be trivial with probability $1$, which agrees with the result in \cite{MR4400944}, but does not agree with \Cref{thm:INTRO prob nontrivial is positive}.

We propose that the assumption that we have $3$ \textit{independent} random maximal isotropic subspaces is the flaw with the na\"ive heuristic above.  We give a precise condition for when one should consider these to be independent, which holds for quadratic twist families but not for the family of all elliptic curves.  In some cases where this condition fails we predict that the distribution converges to the na\"ive model as the number of (distinct) prime divisors of $D$ grows, which is true of the family of all elliptic curves as a consequence of \Cref{thm:INTRO prob nonzero goes to zero}.

\subsection{Outline}
In \Cref{sec:CoresSelmerGroups}, we relate the intersections of $2$-Selmer groups of quadratic twists to Selmer groups given by the corestriction Selmer structures.  Following this, in \Cref{sec:SelmerandBQF} we recall the correspondence between $2$-Selmer elements and binary quartic forms, and refine this correspondence for corestriction Selmer groups.

In \Cref{sec:RecallingBS} we recall some necessary results from the pioneering work of Bhargava--Shankar \cite{MR3272925}.  This is used in \Cref{sec:SelmerBundles}, where we first introduce the notion of a $2$-Selmer bundle, which, loosely speaking, is a continuous collection of pairs $(E,\cL)$ where $E/\QQ$ is an elliptic curve and $\cL$ is a Selmer structure on $E[2]$.  We generalise the work of Bhargava--Shankar on counting average sizes of $2$-Selmer groups \cite{MR3272925} to count the average sizes of the associated Selmer groups $\sel{\cL}(\QQ, E[2])$, and conclude the section by applying this to $\cL=\CCC(K)$ to obtain \Cref{thm:INTRO average size of intersection is product of densities} and \Cref{thm:INTRO prob nontrivial is positive}.

In \Cref{sec:LocalDensities} we compute the local densities in \Cref{thm:INTRO average size of intersection is product of densities}, which allows us to deduce \Cref{thm:INTRO prob nonzero goes to zero} and consequently \Cref{thm:INTROTHM imprecise on decomp}.

In \Cref{sec:Heuristic} we present the heuristic model summarised above for the behaviour of the intersection $\sel{2}(E/\QQ)\cap\sel{2}(E_D/\QQ)$ as $E$ varies in some reasonable family of elliptic curves.

\subsection{Acknowledgements}
We would like to thank Bjorn Poonen for a very helpful conversation during the workshop ``Arithmetic, Algebra, and Algorithms'' at the ICMS in Edinburgh.  We are also grateful to Alex Bartel and Adam Morgan for helpful conversations, as well as comments on a previous version of this manuscript.  Part of this work was supported by a Ph.D. scholarship from the Carnegie Trust for the Universities of Scotland.

\subsection{Notation and Conventions}
\subsubsection*{Fields and Cohomology}
Throughout the article, $F$ will often denote a number field, $\Omega_F$ its set of places, and $F_v$ the completion of $F$ at $v\in\places_F$.  We write $G_F$, $G_{F_v}$ for the corresponding absolute Galois groups, having fixed appropriate algebraic closures $\bar{F},\bar{F_v}$, and embeddings $\bar{F}\to\bar{F_v}$, so that $G_{F_v}\subseteq G_F$.

For a field $K$ as above and a (discrete) $G_K$-module $M$, we write $H^1(K,M)$ for the continuous group cohomology $H^1(G_K, M)$.  Additionally, if $K=F_v$ is a local field, then we write $H^1_\nr(F_v,M)$ for the unramified classes in $H^1(F_v,M)$, meaning those which are trivial on restriction to the inertia subgroup $I_v\leq G_{F_v}$.  For a field extension $L/K$, write
\begin{align*}\cor_{L/K}&:H^1(L, M)\to H^1(K,M),&\res_{L/K}&:H^1(K,M)\to H^1(L,M)\end{align*}
for the natural corestriction and restriction maps.
\subsubsection*{Elliptic Curves and Selmer Structures}
$E/F$ will denote an elliptic curve, and for $n\in\ZZ_{>1}$ we use $E[n]$ to denote the $n$-torsion subgroup.

By a Selmer structure on $E[n]$ (over $F$), we will mean a collection $\cL=(\cL_v)_{v\in\places_F}$ of subgroups $\cL_v\leq H^1(F_v, E[n])$ such that for all but finitely many $v\in\places_F$ we have $\cL_v=H^1_\nr(F_v,E[n])$.  Associated to this data, we have a Selmer group 
\[\sel{\cL}(F,E[n])=\set{x\in H^1(F,E[n])~:~\res_v(x)\in \cL_v\ \forall v}\]
where $\res_v:H^1(F,E[n])\to H^1(F_v,E[n])$ is the restriction map above.  Such groups are always finite by a standard argument, and a brief summary of their properties is given in \cite{MR4400944}*{\S3}.
\subsubsection*{Specific Notation}
Below we provide a table of common notations which, where appropriate, indicates where various notations are introduced.   Excluded from this is the specific notation of \Cref{not:FORGETTHIS}, which is solely for use in \Cref{subsec:LocalTheory}.
\begin{center}
\begin{tabular}{|c|c|c|}
    \hline
    \textbf{Notation}&\textbf{Meaning}&\textbf{Location}\\
    \hline
    $\cG_\cA^+\braces{K}$&Average of genus theory&\Cref{lem:genus theory gives independence}\\
    $E_\theta$&Quadratic twist of $E$ by $\theta$&\Cref{def:quadtwistisom}\\
    $\varphi_\theta$&isomorphism $E_\theta\to E$, defined over field extension by $\sqrt{\theta}$&\Cref{def:quadtwistisom}\\
    $\SSS_v^{(\theta)}(F;E)$&twisted Kummer image in $H^1(F_v,E[2])$&\Cref{def:twistedKummer}\\
    $\CCC_v(K/F;E)$& corestriction Selmer conditions&\Cref{def:cores selmer}\\
    $\sel{\CCC(K)}(F,E[2])$&corestriction Selmer group&\Cref{def:cores selmer}\\
    $V_R,V_{R}^{(i)}$&spaces of binary quartic forms over $R$&\Cref{def:V_Z^i}\\
    $I(E),J(E)$&associated invariants to binary quartic form&\Cref{def:I(E) and J(E)}\\
    $\cL$&$2$-Selmer bundle&\Cref{def:2selbundle}\\
    $\mu_p(S)$&$p$-adic density of set $S\subseteq \ZZ_p^n$&-\\
    $\Epsilon_p$&$p$-adic conditions on family of all elliptic curves&\Cref{not:all EC}\\
    $\Epsilon$&family of all elliptic curves over $\QQ$&\Cref{eq:epsilon}\\
    $\Epsilon(X)$&subset of $\Epsilon$ of height at most $X$&\Cref{eq:epsilon}\\
    $\cF$&a large family of elliptic curves&\Cref{def:large family}\\
    $\Inv(\cF), \Inv_v(\cF)$&Spaces of invariants for $\cF$&\Cref{not:invariants for large family}\\
    $\cG_\cA^+$&upper average of genus theory invariant&\Cref{lem:genus theory gives independence}\\
    \hline
\end{tabular}
\end{center}

We write $\omega(n)$ for the number of distinct prime divisors of $0\neq n\in\ZZ$.
\section{Corestriction Selmer Groups}\label{sec:CoresSelmerGroups}
We begin with an alternative interpretation of intersections of $2$-Selmer groups in terms of the corestriction map.  This will enable us to perform density computations later in the paper.  Throughout the section, we let $F$ be a number field.

\begin{definition}\label{def:quadtwistisom}
    Let $E/F$ be an elliptic curve.  Recall that for each $\theta\in F^\times/F^{\times2}$ there is an $F(\sqrt{\theta})$ isomorphism $E_\theta\cong E$, where the former is the quadratic twist of $E$ by $\theta$.  If we are given Weierstrass equations
    \begin{align*}
        E&:y^2=x^3+Ax+B\\
        E_\theta&:\theta y^2=x^3+Ax+B
    \end{align*}
    then the map is
    \begin{align*}
        E_\theta&\to E\\
        (x,y)&\mapsto (x,\sqrt{\theta}y)
    \end{align*}

    If $\theta$ is not the trivial class, then clearly this isomorphism maps the points of $E_\theta(F)$ exactly to those of $E(F(\sqrt{\theta}))$ which are acted on by $-1$ by the generator of $\gal(F(\sqrt{\theta})/F)$, and so in particular it restricts to an isomorphism of $G_F$-modules 
    \[\varphi_\theta:E_\theta[2]\cong E[2].\]
\end{definition}
\subsection{Twisted Kummer Images}\label{subsec:TwistedKummerImages}
The map $\varphi_\theta$ gives rise to various twisted Kummer images locally, which will cut out $\sel{2}(E_\theta/\QQ)\subseteq H^1(F, E[2])$.

\begin{definition}\label{def:twistedKummer}
    For elliptic curve $E/F_v$ and each $\theta\in F_v^\times/F_v^{\times2}$, we write $\SSS_v^{(\theta)}(F;E)$ for the associated twisted Kummer image.  That is, the image of
    \[\begin{tikzcd}
    E_\theta(F_v)/2E_\theta(F_v)\arrow[r, "\delta_\theta"]&
    H^1(F_v, E_\theta[2])\arrow[r, "(\varphi_\theta)^*" above, "\sim" below]&
    H^1(F_v, E[2]),
    \end{tikzcd}\]
    where $\delta_\theta$ is the usual connecting map from the short exact sequence induced by multiplication by $2$ on $E_\theta$.  When $\theta$ is the trivial class, we will abbreviate $\SSS_v(F;E):=\SSS_v^{(1)}(F;E)$.
\end{definition}
\noindent With respect to the local Tate pairing, these local groups are self-dual.
\begin{lemma}\label{lem:twisted Kummer image is selfdual}
    For every place $v$ of $F$, each elliptic curve $E/F_v$ and each $\theta\in F_v^\times/F_v^{\times2}$, the twisted Kummer image $\SSS_v^{(\theta)}(F;E)\subseteq H^1(F_v, E[2])$ is its own orthogonal complement with respect to the local Tate pairing.
\end{lemma}
\begin{proof}
    Note that since $\varphi_\theta$ is an $F_v(\sqrt{\theta})$-isomorphism, the Weil pairing on $E_\theta[2]$ is preserved by $\varphi_\theta$ and so this follows from the case when $\theta$ is the trivial class which, in turn, follows from Tate's local duality \cite{MR0175892}*{Theorem 2.3}.
\end{proof}

We now note that these local groups do, in fact, define Selmer structures.
\begin{proposition}
    Let $E/F$ be an elliptic curve, $\theta\in F^\times/F^{\times2}$, and consider the collection $\SSS^{(\theta)}(F;E):=\set{\SSS_v^{(\theta)}(F;E)}_{v\in\places_F}$.  Then for all but finitely many $v\in\places_F$ we have
    \[\SSS_v^{(\theta)}=H^1_\nr(F,E[2]).\]
    In particular $\SSS^{(\theta)}(F,E[2])$ defines a Selmer structure for $E[2]$, and moreover
    \[\sel{\SSS^{(\theta)}}(F,E[2])=\varphi_\theta^*\braces{\sel{2}(E_\theta/F)}.\]
\end{proposition}
\begin{proof}
    The only thing which needs to be checked is the first claim, as the rest is immediate.

    For all but finitely many $v\in\places_F$ we have two equalities
    \begin{align*}
        \SSS_v(F;E)&=H^1_{\nr}(F_v, E[2]),& \SSS_v^{(\theta)}(F;E)&=\varphi_\theta^*\left(H^1_{\nr}(F_v, E_\theta[2])\right).
    \end{align*}
    Since Galois isomorphisms of modules map unramified classes to unramified classes, $\SSS_v(F;E)=\SSS_v^{(\theta)}(F;E)=H^1_{\nr}(F_v, E[2])$ for such $v$
\end{proof}

\subsection{Local Theory}\label{subsec:LocalTheory}
We take some new notation for the rest of this section.
\begin{notation}\label{not:FORGETTHIS}
    Let $K/F$ be a multiquadratic extension.  Let $v\in\places_F$, and assume that $w\in\places_K$ is a place extending $v$.  Write $G_v:=\gal(K_w/F_v)$.  Moreover, assume that the local degree is $[K_w:F_v]=2^r$, write $K_w=F_v(\sqrt{\theta_1},\dots,\sqrt{\theta_r})$ for some $\theta_i\in F_v$, and let $S:=\gp{\theta_1,\dots,\theta_r}\leq F_v^\times/F_v^{\times 2}$.  We denote $K_{w,i}:=F_v(\sqrt{\theta_i})$.  Let $E/F_v$ be an elliptic curve with a choice of minimal discriminant $\Delta_{E}$.  
\end{notation}

We now prove some useful lemmata, before going on to provide the main proposition for this subsection.
\begin{lemma}\label{lem:Galstructs for tors}
    With notation as in \Cref{not:FORGETTHIS}, exactly one of the following holds.
    \begin{itemize}
        \item[1.] $E(K_w)[2]=E(F_v)[2]$.
        \item[2.] $\Delta_E\in K_w^{\times2}\backslash F_v^{\times 2}$ and, writing $H_v=\gal(K_w/F_v(\sqrt{\Delta_E}))$, there is an isomorphism of $\FF_2[G_v]$--modules $E(K_w)[2]\cong_{\FF_2[G_v]}\FF_2[G_v/H_v]$.
    \end{itemize}
\end{lemma}
\begin{proof}
    The nontrivial $2$-torsion points on $E$ over the algebraic closure are precisely the points which have $x$-coordinates which are roots of the cubic polynomial $f(x)$.  In particular, the only situation in which $E(F_v)[2]\neq E(K_w)[2]$ must be when $\dim_{\FF_2}E(F_v)[2]=1$ and $\dim_{\FF_2}E(K_w)[2]=2$.  In this latter case, the extra $2$-torsion point is obtained over the quadratic extension $F_v(\sqrt{\Delta_E})$.  Then $G_v$ acts on $E(K_w)[2]$ via the quotient $G_v/H_v$, and this module must be $2$-dimensional with a $1$-dimensional fixed space and so (by e.g. \cite{alperin_1986}*{page 24}) there is precisely one isomorphism class of module to which it can belong:  that of $\FF_2[G_v/H_v]$.
\end{proof}

\begin{lemma}\label{lem:H1ofMultiQuadIsSumOfQuadH1}
    With notation as in \Cref{not:FORGETTHIS}, we have that
    \[H^1(K_w/F_v, E(K)[2])=\oplus_{i=1}^r\infl_i\left(H^1\left(K_{w,i}/F_v, E(K_{w,i})[2]\right)\right)\]
    where the direct sum here is the internal sum of $\FF_2$-vector spaces, and 
    \[\inf_i:H^1(K_{w,i}/F_v, E(K_{w,i})[2])\to H^1(K_w/F_v, E(K_{w})[2])\]
    denotes the usual inflation map from $K_{w,i}$ to $K_w$.
\end{lemma} 
\begin{proof}
    Without loss of generality, if $F_v(\sqrt{\Delta_E})/F_v$ is a quadratic extension contained in $K$, then we assume $\theta_1=\Delta_E$.

    If $E(F_v)[2]=E(K_w)[2]$ then the $\gal(K_w/F_v)$ action is trivial on $E(K_w)[2]$, and so the cohomology groups in the statement are just $\hom(\gal(K_w/F_v), E(K_w)[2])$ and similarly for the $K_{w,i}$.  Moreover the inflation maps are simply given by extension of homomorphisms, and so the result is clear.

    By \Cref{lem:Galstructs for tors}, it remains to consider the case that $\Delta_E\in K_w^{\times2}\backslash F_v^{\times2}$ and \newline$E(K_w)[2]\cong_{\FF_2[G_v]} \FF_2[G_v/H_v]$ where 
    \[G_v:=\gal(K_w/F_v)\geq \gal(K_w/K_{w,1})=\gal\left(K_w/F_v(\sqrt{\Delta_E})\right)=:H_v.\]
    In this case we have a commutative diagram for each $2\leq i\leq r$
    \[\begin{tikzcd}
        H^1(K_w/F_v,\ E(K_w)[2])\ar{r}{\sim}&H^1(K_w/K_{w,1},\ \FF_2)\\
        H^1(K_{w,i}/F_v,\ E(K_{w,i})[2])\ar{r}{\sim}\ar{u}{\infl_i}&H^1(K_{w,i}\cdot K_{w,1}/K_{w,1},\ \FF_2)\ar{u}{\infl_i},
    \end{tikzcd}\]
    where the vertical maps are the natural inflation maps, the lower horizontal is induced by the natural isomorphism $\gal(K_{w,i}\cdot K_{w,1}/K_{w,1})\cong \gal(K_{w,i}/F_v)$ and the fact that $E(K_{w,i})[2]=E(F_v)[2]\cong \FF_2$, and the top horizontal is the isomorphism of Shapiro's lemma (as in, for example, \cite{neukirch2013class}*{Theorem 4.19}).  That this diagram commutes is clear from the definitions of the maps.  Clearly (as in the previous case), we have
    \[H^1(K_w/K_{w,1},\ \FF_2)= \oplus_{i=2}^r\infl_{i}\left(H^1(K_{w,i}\cdot K_{w,1}/K_{w,1},\ \FF_2)\right),\]
    since both are just groups of homomorphisms, and so in particular by the diagram above
    \[H^1(K_w/F_v,\ E(K_w)[2])\cong \oplus_{i=2}^r\infl_i\left(H^1(K_{w,i}/F_v,\ E(K_{w,i})[2])\right).
    \]
    Noting that $E(K_w)[2]\cong_{\FF_2[G_v]} \FF_2[G_v/H_v]$ we have from Shapiro's Lemma that
    \[H^1(K_{w,1}/F_v, E(K_{w,1})[2])=0,\]
    and so the result follows.
\end{proof}

\begin{lemma}\label{lem:points that are sums of the quadratic points}
    Set notation as in \Cref{not:FORGETTHIS}.  Let $\delta:E(K_w)\to H^1(K_w, E[2])$ be the connecting map arising from the multiplication by $2$ exact sequence on $E$.  For every point $P\in E(K_w)$ such that $\delta(P)\in \textup{Im}(\res_{K_w/F_v}:H^1(F_v, E[2])\to H^1(K_w, E[2]))$, there is a tuple $(P_\theta)\in \prod_{\theta\in S}E_\theta(F_v)$ such that
    \[P=\sum_{\theta\in S}\varphi_{\theta}(P_\theta).\]
\end{lemma}
\begin{proof}
    We will prove this by induction on $r=\dim_{\FF_2}\gal(K_w/F_v)$.  If $r=0$ then the claim is trivial.  Assuming that $r>0$, let $M$ be an intermediate field such that $[K_w:M]=2$ and write $K_w=M(\sqrt{\theta})$ for some $\theta\in S$.  Let $x\in H^1(F_v, E[2])$ be such that $\delta(P)=\res_{K_w/F_v}(x)$. Observe that the commutative diagram
    \[\begin{tikzcd}
        E(K_w)\ar{r}{\delta}\ar{d}{N_{K_w/M}}&H^1(K_w, E[2])\ar{d}{\cores_{K_w/M}}\\
        E(M)\ar{r}{\delta}&H^1(M, E[2]),
    \end{tikzcd}\]
    shows that 
    \[\delta(N_{K_w/M}(P))=\cores_{K_w/M}\circ\res_{K_w/M}\circ\res_{M/F_v}(x)=2\res_{M/F_v}(x).\]
    Since $H^1(M, E[2])$ is a $2$-torsion group we have that $N_{K_w/M}(P)\in \ker(\delta)=2E(M)$, so let $Q\in E(M)$ be a point such that $2Q=N_{K_w/M}(P)$.  Let $R:=P-Q$, and let $\sigma\in \gal(K_w/M)$ be the nontrivial element, and observe that
    \[\sigma(R)=\sigma(P)-Q=N_{K_w/M}(P)-P-Q=Q-P=-R.\]
    In particular $R=\varphi_\theta(P_\theta)$ for some point $P_\theta\in E_\theta(M)$.  Now since $P_\theta$ and $Q$ are both points on elliptic curves satisfying the constraints above (and since $\varphi_{\theta}\circ\varphi_{\theta'}=\varphi_{\theta\theta'}$) but with $K_w/F_v$ replaced by the degree $2^{r-1}$ extension $M/F_v$, we conclude the result.
\end{proof}

These lemmata then allow us to prove the following very useful result on the corestriction and the preimage under restriction of the Kummer image.

\begin{proposition}\label{prop:local corestriction images are intersections of twisted Kummers}
    Take the notation of \Cref{not:FORGETTHIS}.  Then we have identities:
    \begin{align*}
    \res_{K_w/F_v}^{-1}(\SSS_w(K;E))&=\sum_{\theta\in S}\SSS_v^{(\theta)}(F;E),\\
    \cores_{K_w/F_v}(\SSS_w(K;E))&=\bigcap_{\theta\in S}\SSS_v^{(\theta)}(F;E).
\end{align*}    
\end{proposition}
\begin{proof}
    To ease notation, we write $\FFF_v:=\res_{K_w/F_v}^{-1}(\SSS_w(K;E))$ and \newline$\CCC_v:=\cores_{K_w/F_v}(\SSS_w(K;E))$.  Firstly we note that the property for $\CCC_v$ follows from that of $\FFF_v$, the fact that these two are dual with respect to the local Tate pairing and that $\SSS_v^{(\theta)}(F;E)$ is its own orthogonal complement (\Cref{lem:twisted Kummer image is selfdual}).

    It remains to prove the identity for $\FFF_v$, which we now do.  If $\theta\in S$ then it is easy to see that the following diagram commutes:
    \begin{equation}\label{eq:PointsAndCohomForNiceExtensionFIdent}
    \begin{tikzcd}
        E_\theta(F_v)\ar{r}{\delta}\ar{d}{\varphi_\theta}&H^1(F_v, E_\theta[2])\ar{r}{\varphi_\theta^*}&H^1(F_v, E[2])\ar{d}{\res_{K_w/F_v}}\\
        E(F_v(\sqrt{\theta}))\ar{r}&E(K_w)\ar{r}{\delta}&H^1(K_w, E[2]),
    \end{tikzcd}
    \end{equation}
    and so $\SSS_v^{(\theta)}(F;E)\subseteq\FFF_v$.

    It remains to prove the other inclusion.  Let $x\in \FFF_v$, and choose $P\in E(K_w)$ such that $\delta(P)=\res_{K_w/F_v}(x)$.  Note that by \Cref{lem:points that are sums of the quadratic points} we can write
    \[P=\sum_{\theta\in S}\varphi_\theta(P_\theta)\]
    for some points $P_\theta\in E_\theta(F_v)$.  Then by commutativity of \eqref{eq:PointsAndCohomForNiceExtensionFIdent} we have
    \begin{align*}
    \res_{K_w/F_v}(x)
    =\delta(P)
    &=\sum_{\theta\in S}\delta\circ\varphi_\theta(P_\theta)
    \\&=\res_{K_w/F_v}\left(\sum_{\theta\in S}\varphi_\theta^*\circ\delta(P_\theta)\right)\in \res_{K_w/F_v}\left(\sum_{\theta\in S}\SSS_v^{(\theta)}(F;E)\right).
    \end{align*}
    It is then sufficient to prove that $\ker(\res_{K_w/F_v})\subseteq \sum_{\theta\in S}\SSS_v^{\theta}(F;E)$ when $K_w/F_v$ is a multiquadratic local extension, which we now do.  Using the inflation-restriction exact sequence, \Cref{lem:H1ofMultiQuadIsSumOfQuadH1} and compatibility of inflation maps, there is an equality
    \begin{align*}
    \ker(\res_{K_w/F_v})
    &=\im\left(\infl:H^1(K_w/F_v, E(K_w)[2])\to H^1(F_v, E[2])\right)
    \\&=\sum_{\theta\in S}\im\left(\infl:H^1\left(F_v(\sqrt{\theta})/F_v, E(F_v(\sqrt{\theta}))[2]\right)\to H^1\left(F_v, E[2]\right)\right)
    \\&=\sum_{\theta\in S}\ker\left(\res_{F_v(\sqrt{\theta})/F_v}\right).
    \end{align*}
    By \cite{kramer1981arithmetic}*{Lemma 3}, for each $\theta\in S$
    \[\ker\left(\res_{F_v(\sqrt{\theta})/F_v}\right)\subseteq \SSS_v^{}(E)+\SSS_v^{(\theta)}(E),\]
    and so the required result holds.
\end{proof}
\begin{rem}
    In the case that $K/F$ is a quadratic extension, \Cref{prop:local corestriction images are intersections of twisted Kummers} was shown by Kramer \cite{kramer1981arithmetic}*{Proposition 7}, the novelty here is in proving the local results for multiquadratic extensions.
\end{rem}

\subsection{Global Theory}\label{subsec:GlobalTheory}
We now aim to use the local results of the previous subsection to describe intersections of Selmer groups of twists in terms of the corestriction Selmer structure.  We begin by recalling the definition, as it is presented in \cite{Paterson2021}*{Definition 4.6, Lemma 4.8} (see also \cites{kramer1981arithmetic,MR4400944} for the quadratic case).

\begin{definition}\label{def:cores selmer}
    Let $E/F$ be an elliptic curve, and $K/F$ be a multiquadratic extension. For each $v\in\places_F$, choose a place $w\in\places_K$ extending it and define
    \[\CCC_v(K/F;E):=\frac{N_{K_w/F_v}E(K_w)+2E(F_v)}{2E(F_v)}\subseteq H^1(F_v, E[2]),\]
    where the inclusion is induced by the short exact sequence $0\to E[2]\to E\to E\to 0$.  Moreover, define the corestriction Selmer group to be
    \[\sel{\CCC(K)}(F,E[2])=\set{x\in H^1(F, E[2])~:~\res_v(x)\in \CCC_v(K/F;E)\ \forall v\in\places_F},\]
    where $\res_v:H^1(F,E[2])\to H^1(F_v,E[2])$ is the restriction map.
\end{definition}

It is worth briefly remarking on a special set of local conditions, which we make use of later.

\begin{lemma}\label{for good I0I1 unram have cores is sel}
    Let $E/F$ be an elliptic curve, and $K/F$ be a multiquadratic extension.  Then if $v\in\places_F$ is unramified in $K/F$ and $E/F_v$ has reduction of type $I_0$ or $I_1$ then
    \[\CCC_v(K/F;E)=\SSS_v(F,E).\]
\end{lemma}
\begin{proof}
    By \cite{mazur1972rational}*{Cor 4.4} and \cite{Paterson2021}*{Remark 3.5} we have that 
    \[E(F_v)=N_{K_w/F_v}E(K_w),\]
    from which the claimed identity is immediate.
\end{proof}

\begin{proposition}\label{prop:Multiquadratic extn desc of CandF with Twisted Kummers}
Assume that $F$ is a number field, $K/F$ is a multiquadratic extension, and let $E/F$ be an elliptic curve.  Let $v\in\places_F$ and $w\in\places_K$ be a place extending $v$, and $S:=\ker(F^\times/F^{\times2}\to K^\times/K^{\times2})$.  Then
\begin{align*}
    \FFF_v(K/F;E)&=\sum_{\theta\in S}\SSS_v^{(\theta)}(F;E),\\
    \CCC_v(K/F;E)&=\bigcap_{\theta\in S}\SSS_v^{(\theta)}(F;E).
\end{align*}
Moreover,
\[\sel{\CCC(K)}(F, E[2])=\bigcap_{\substack{\theta\in S}}(\varphi_\theta)^*\left(\sel{2}(E_\theta/F)\right).\]
\end{proposition}
\begin{proof}
    The first two identities follow from \Cref{prop:local corestriction images are intersections of twisted Kummers}.  The equality for the Selmer group then follows from the equalities for the local groups $\set{\CCC_v(K/F:E)}_{v\in\places_F}$.
\end{proof}

\section{Selmer Elements via Binary Quartic Forms}\label{sec:SelmerandBQF}
We first discuss the explicit correspondence between $2$-Selmer elements and certain equivalence classes of binary quartic forms.  This explicit version of the correspondence was first discovered by Birch and Swinnerton-Dyer \cite{MR146143} and has later been developed substantially \cite{MR1628193} (see also \cites{MR1693411,MR2509048,MR1951757}).  We conclude this section by determining which equivalence classes of binary quartic forms correspond to elements in the intersections of $2$-Selmer groups of quadratic twists.

\subsection{Binary Quartic Forms and Selmer Elements}
Recall that a binary quartic form over a ring $R$ is a degree $4$ homogeneous polynomial in two variables with coefficients in $R$.

\begin{notation}\label{def:V_Z^i}
For each ring $R$ we define $V_R$ to be the set of binary quartic forms with coefficients in $R$.  Moreover, for $i=0,1,2$ denote by $V_\RR^{(i)}$ to be the subset of binary quartic forms in $V_\RR$ with nonzero discriminant having $i$ pairs of complex roots in $\mathbb{P}^1_\CC$ and $4-2i$ real roots in $\mathbb{P}^1_\CC$.  The set of definite quartic forms, those which take only positive or only negative values when evaluated at nonzero elements $(x_0,y_0)\in\RR^2$, is $V_{\RR}^{(2)}$.  We denote the subset of positive (resp. negative) definite forms by $V_{\RR}^{(2+)}$ (resp. $V_{\RR}^{(2-)}$).  We write $V_\ZZ^{(i)}:=V_\ZZ\cap V_\RR^{(i)}$.
\end{notation}

\begin{definition}
    Let $K$ be a field and $g(x,y)\in V_K$ be a binary quartic form.  Then we say that $f$ is $K$-soluble if there are $x,y,z\in K$ with $(x,y)\neq (0,0)$ such that
    \[z^2=g(x,y).\]
    We say that $g\in V_\QQ$ is locally soluble if $g$ is $\QQ_v$-soluble for every place $v\in\places_\QQ$.
\end{definition}

For every ring $R$ there is a well defined (twisted) action of $\GL_2(R)$ on $V_R$, given for each $g\in V_R$ and $\gamma\in \GL_2(R)$ by
\[\gamma\cdot g(x,y)=\det(\gamma)^{-2}g((x,y)\cdot \gamma^t),\]
where the action of $\gamma$ on $(x,y)$ on the right hand side is just the standard matrix operation on a row vector, and $\gamma^t$ is the transpose matrix.  Moreover, the centre of $\GL_2(R)$ acts trivially, and so this descends to an action of $\PGL_2(R)$.

Associated to $g(x,y)=ax^4+bx^3y+cx^2y^2+dxy^3+ey^4\in V_R$ there are then two invariants of this action, $I(g)$ and $J(g)$, given by
\begin{align*}
    I(g)&=12ae-3bd+c^2,\\
    J(g)&=72ace+9bcd-27ad^2-27eb^2-2c^3.
\end{align*}
There are also the covariants $g_4(g;X,Y), g_6(g;X,Y)\in R[X,Y]$, given by
\begin{eqnarray*}
    g_4(g;X,Y)&=&(3b^2-8ac)X^4 + 4(bc-6ad)X^3Y + 2(2c^2-24ae-3bd)X^2Y^2
    \\&& + 4(cd-6be)XY^3 + (3d^2-8ce)Y^4,
    \\g_6(g;X,Y)&=&b^3+8a^2d-4abc)X^6+2(16a^2e+2abd-4ac^2+b^2c)X^5Y
    \\&&+5(8abe+b^2d-4acd)X^4Y^2+20(b^2e-ad^2)X^3Y^3
    \\&&-5(8ade+bd^2-4bce)X^2Y^4-2(16ae^2+2bde-4c^2e+cd^2)XY^5
    \\&&-(d^3+8be^2-4cde)Y^6.
\end{eqnarray*}

\subsection{Correspondence}

We now describe the relationship to $2$-Selmer groups of elliptic curves.  We note the parts of use to us as one large theorem, to gather everything in one place for use later in the article.  The theorem summarises \cite{MR2509048}*{\S6}, though the reader should note that we consider the elliptic curve $E:y^2=x^3-\tfrac{I}{3}x-\tfrac{J}{27}$, whereas the authors in loc. cit. consider the model (for the same curve) $y^2=x^3-27Ix-27J$, and so there is a change of variables applied between their results and what is stated below.

\begin{theorem}[\cite{MR2509048}*{\S6}]\label{thm:CremonaFisherSummary}
    Given the data of a pair, $(E,F)$ where $F$ is a characteristic $0$ field, and $E/F$ is an elliptic curve with specified Weierstrass equation
    \[E:y^2=x^3-\tfrac{I}{3}x-\tfrac{J}{27},\]
    the following is true.

    Let $\bar{L}:=\bar{F}\times \bar{F}\times \bar{F}$, let $f(X):=X^3-3IX+J$, and consider the \'etale algebra $L=L(E):=F[\epsilon]/\gp{f(\epsilon)}$.  Fix $\epsilon_1,\epsilon_2,\epsilon_3\in \bar{F}$ to be the three roots of $f(X)$ and so fix the embedding $L\subseteq \bar{L}$ induced by mapping $\epsilon\mapsto (\epsilon_1,\epsilon_2,\epsilon_3)$.  Under this embedding, we define the norm of an element $w=(w_1,w_2,w_3)\in L\subseteq \bar{L}$ to be $N_{L/F}w:=w_1w_2w_3\in F$.

    Then there is a commutative diagram (of groups)
    \begin{equation}\label{eq:CremonaFisherDiagram}
    \begin{tikzcd}
        E(F)/2E(F)\ar[r, "\delta"]\ar[d, "q"]
        &H^1(F, E[2])\ar[d, "\alpha"]\\
        \set{\substack{\PGL_2(F)\textnormal{-equivalence classes of}\\\textnormal{binary quartic forms }g\in V_F\\\textnormal{with }(I(g),J(g))=(I,J)}}\ar[r, "z"]\ar[ur, "\delta'"]
        &\ker\left(N_{L/F}:L^\times/L^{\times2}\to F^\times/F^{\times2}\right)
    \end{tikzcd}
    \end{equation}
    where the definitions and properties of the maps are as follows:
    \begin{itemize}
        \item The map $\delta$ is an injection.  It is the usual connecting map arising from taking Galois cohomology on the multiplication by $2$ short exact sequence for $E$.
        \item The map $q$ is injective, and its image is the collection of cosets represented by $F$-soluble binary quartic forms.  For $P=(\xi,\eta)\in E(F)$ we define 
        \[q(P):=\set{X^4-\frac{3}{2}\xi X^2Y^2-\eta XY^3+(\tfrac{I}{12}-\tfrac{3\xi^2}{4})Y^4}\]
        where the braces denote that we map to the $\PGL_2(F)$-equivalence class of the stated binary quartic form.  
        \item The map $\alpha$ is an isomorphism.  This map is constructed as follows (see \cite{MR1326746}*{Theorem 1.1}): consider the map
        \begin{align*}
            w:E[2]&\to \mu_2(\bar{L})\\
            P&\mapsto \left(e_2(P, (\epsilon_1,0)),\ e_2(P, (\epsilon_2,0)),\ e_2(P, (\epsilon_3,0))\right)
        \end{align*}
        where $e_2$ denotes the Weil Pairing.  This induces a map
        \[w^*:H^1(F, E[2])\to H^1(F, \mu_2(\bar{L})).\]
        Additionally we have the Kummer isomorphism (see \cite{MR554237}*{p.152, Ex.2})
        \[\kappa:H^1(F, \mu_2(\bar{L}))\cong L^\times/L^{\times 2}.\]
        Then $\alpha:=\kappa\circ w^*$.
        \item The map $z$ called the cubic invariant, and is injective.  It is constructed as follows.  For a binary quartic form $g$ representing an element of the domain, define the `irrational covariant' to be
        \[G(X,Y):=\frac{1}{3}\left(4\epsilon g(X,Y)+g_4(g;X,Y)\right)\in L[X,Y].\]
        Choose $(x,y)\in F\times F$ such that $G(x,y)\in L^\times$ (such a pair exists by \cite{MR2509048}*{\S2 Paragraph 2}), and define
        \[z(g):=G(x,y)\in L^\times/L^{\times2}.\]
        This map is independent of the choice of $(x,y)$ by \cite{MR2509048}*{Proposition 2}, and $z(g)$ has square norm by \cite{MR2509048}*{Lemma 1}.
        \item The map $\delta'$ is injective, and is constructed as follows.  Let $g(X,Y)$ be a representative of an equivalence class in the domain.  Then let $C$ be the smooth projective curve with affine equation $Z^2=g(X,1)$, and define the map
        \begin{align*}
        \pi:C&\to E\\
        (x,z)&\mapsto \left(\frac{g_4(g;x,1)}{12z^2}, \frac{g_6(g;x,1)}{8z^3}\right).
        \end{align*}
        The pair $(C,\pi)$ is in fact a $2$-covering of $E$ (see \cite{MR2509048}*{\S6 p11 Remarks (3)} and references therein).  The map $\delta'$ sends the equivalence class of $g$ to the equivalence class of the $2$-covering $(C,\pi)$ and then maps this to a cocycle class via the usual correspondence (see e.g. \cite{stoll2006descent}).
    \end{itemize}
\end{theorem}

From this we state a shorter helpful corollary, which is at the core of the work of Bhargava and Shankar on $2$-Selmer groups \cite{MR3272925}*{see e.g. \S3.1}.
\begin{corollary}[See also \cite{MR3272925}*{Theorem 3.2} and references therein]\label{cor:solBQFsAreE/2}
    Let $F$ be a characteristic $0$ field, and $E/F$ be an elliptic curve with specified Weierstrass equation 
    \[E:y^2=x^3-\tfrac{I}{3}x-\tfrac{J}{27}.\]

    Using the notation of \Cref{thm:CremonaFisherSummary} for the data $(E,F)$ above, the map $q$ is a bijection between $\PGL_2(F)$-orbits of $F$-soluble binary quartic forms with invariants $I,J$ and elements of $E(F)/2E(F)$.  Under this bijection, the identity element corresponds to the (unique!) $\PGL_2(F)$-orbit of binary quartic forms possessing a linear factor over $F$.

    Furthermore, the stabiliser in $\PGL_2(F)$ of any (not necessarily $F$-soluble) binary quartic form $g\in V_F$ with invariants $(I,J)=(I(g),J(g))$ such that $4I^3-J^2\neq 0$ is isomorphic to $E(F)[2]$ where $E$ is the elliptic curve defined by $y^2=x^3-\tfrac{I}{3}x-\tfrac{J}{27}$.
\end{corollary}

\subsection{Elliptic Curves Over \texorpdfstring{$\QQ$}{the Rational Numbers}}

At this point we restrict to our case of interest, which here is elliptic curves over $F=\QQ$.  We begin with a useful bookkeeping definition.

\begin{definition}[\cite{MR3272925}*{\S3}]\label{def:I(E) and J(E)}
    Let $E/\QQ$ be an elliptic curve, and let $(A,B)\in\Epsilon$ be the unique pair such that
    \begin{equation}\label{eq:E_(A,B)}
        E\cong E_{A,B}:y^2=x^3+Ax+B.
    \end{equation}
    Then we define the quantities
    \begin{align*}
        I(E)&:=-3A\\
        J(E)&:=-27B
    \end{align*}
    Moreover, for $(I,J)=(I(E),J(E))$ we will use the notation $E^{I,J}:y^2=x^3-\tfrac{I}{3}x-\tfrac{J}{27}$ in order to pass back from the invariants $I(E),J(E)$ to a model for the curve $E$.
\end{definition}

We then have the following well-known proposition, which follows from \cite{MR146143}*{Lemma 2} (see also \cite{MR1628193}).

\begin{proposition}[see also \cite{MR3272925}*{Proposition 3.3}]\label{prop:2SelmerElementsAsBQF}
    Let $E/\QQ$ be an elliptic curve, with specified Weierstrass equation
    \[E:y^2=x^3-\tfrac{I}{3}x-\tfrac{J}{27}.\]
    Then using the notation of \Cref{thm:CremonaFisherSummary} for the data $(E,\QQ)$: the map $\delta'$ is a bijection between the set of $\PGL_2(\QQ)$-orbits of locally soluble binary quartic forms $g\in V_\QQ$ with $(I(g),J(g))=(I,J)$ and the $2$-Selmer group $\sel{2}(E/\QQ)\subseteq H^1(\QQ, E[2])$.

    Furthermore, the set of binary quartic forms $g\in V_\QQ$ having a linear factor (over $\QQ$) and invariants $(I,J)$ lie in a single $\PGL_2(\QQ)$-orbit, and this orbit maps to the identity element of $\sel{2}(E/\QQ)$.
\end{proposition}

In order to reduce to counting lattice points, we have the following lemma which follows from \cite{MR146143}*{Lemmas 3,4,5} and shows that we can always find an integral representative (i.e. $g\in V_\ZZ\subset V_\QQ$) in the $\PGL_2(\QQ)$-orbit of a locally soluble binary quartic form.
\begin{lemma}[\cite{MR146143}*{Lemmas 3,4,5}, see also \cite{MR3272925}*{Lemma 3.4}]\label{lem:loc sol BQF with decent invts are repd by integral BQF}
    Let $g\in V_\QQ$ be a locally soluble binary quartic form having integer invariants $(I,J):=(I(g), J(g))$ such that $(2^4\cdot 3)\mid I$ and $(2^6\cdot 3^3)\mid J$.  Then $g$ is $\PGL_2(\QQ)$-equivalent to an element of $V_\ZZ$.
\end{lemma}

Since each elliptic curve $E/\QQ$ is isomorphic to the elliptic curve defined by the equation $y^2=x^3-\tfrac{2^4I(E)}{3}x-\tfrac{2^6J(E)}{27}$, \Cref{lem:loc sol BQF with decent invts are repd by integral BQF} and \Cref{prop:2SelmerElementsAsBQF} imply the following key theorem.

\begin{theorem}[\cite{MR3272925}*{Theorem 3.5}]\label{thm:Integral BQF with proper invariants correspond to selmer elements}
    Let $(A,B)\in\Epsilon$, and specify the Weierstrass equation for $E\cong E_{A,B}$ to be
    \[E:y^2=x^3+2^4Ax+2^6B.\]
    Then using the notation of \Cref{thm:CremonaFisherSummary} for the data $(E,\QQ)$: the map $\delta'$ induces a bijection between $\PGL_2(\QQ)$-equivalence classes of locally soluble binary quartic forms $g\in V_\ZZ$ with invariants $(I(g),J(g))=(2^4I(E),2^6J(E))$ and the $2$-Selmer group $\sel{2}(E/\QQ)$.

    Furthermore, the set of $g\in V_\ZZ$ having a linear factor (over $\QQ$) and invariants $(2^4I(E),2^6J(E))$ lie in a single $\PGL_2(\QQ)$-orbit, and this orbit maps to the identity element of $\sel{2}(E/\QQ)$.
\end{theorem}

We now improve \Cref{cor:solBQFsAreE/2} (in the case that $F$ is a local field) to determine which locally soluble binary quartic forms correspond to intersections of the groups $\SSS_v^{(\theta)}$.  We will then improve \Cref{thm:Integral BQF with proper invariants correspond to selmer elements}, to also describe the corestriction Selmer group in terms of binary quartic forms.

\begin{definition}
Let $F$ be a characteristic $0$ field, and $S\leq F^\times/F^{\times2}$ be a finite subgroup.  Then we say that a binary quartic form $g(X,Y)\in V_F$ is $(F,S)$-soluble if for each $\theta\in S$ there exist $x,y,z\in F$ with $(x,y)\neq (0,0)$ such that
\[z^2=\theta g(x,y).\]
Moreover, we say that $g\in V_\QQ$ is locally $S$-soluble if for every place $v\in\places_\QQ$, $g$ is $(\QQ_v, S)$-soluble ($S$ here being interpreted as its image in $\QQ_v^\times/\QQ_v^{\times2}$).
\end{definition}
\begin{rem}
    This is clearly well defined, since altering $\theta$ by a square scales the chosen $z$-coordinate.  Further note that being $(F,\gp{1})$-soluble is the same as being $F$-soluble, and in fact being $(F,S)$-soluble always requires at least being $F$-soluble. 
\end{rem}

Before we can describe image of $(F,S)$-soluble binary quartic forms under the correspondence of \Cref{cor:solBQFsAreE/2}, we will require a helpful lemma.

\begin{lemma}\label{thm:BQF rep of twisted Kummer}
    Let $F$ be a characteristic $0$ field, and let $E/F$ be an elliptic curve with fixed Weierstrass equation $E:y^2=x^3+\tfrac{I}{3}x-\tfrac{J}{27}$.  Let $\theta\in F^\times$ be nonsquare, and fix the following model for the quadratic twist $E_{\theta}$
    \[E_{\theta}:y^2=x^3-\tfrac{\theta^2 I}{3}x-\tfrac{\theta^3J}{27}.\]
    Write $\varphi_\theta:E_{\theta}\to E$ for the isomorphism (over $F(\sqrt{\theta})$) given by $(x,y)\mapsto (\frac{x}{\theta}, \frac{y}{\theta\sqrt{\theta}})$.  This map restricts to an isomorphism (over $F$) $E_{\theta}[2]\to E[2]$, and the diagram below commutes:
    \[\begin{tikzcd}
        H^1(F, E_{\theta}[2])\ar[r, "\varphi_\theta^*"]&H^1(F, E[2])\\
        \set{\substack{\PGL_2(F)\textnormal{-equivalence classes of}\\\textnormal{binary quartic forms }g\in V_F\\\textnormal{with }(I(g),J(g))=(\theta^2I,\theta^3J)}}\ar[r, "\phi_\theta"]\ar[u, "\delta'"]
        &\set{\substack{\PGL_2(F)\textnormal{-equivalence classes of}\\\textnormal{binary quartic forms }g\in V_F\\\textnormal{with }(I(g),J(g))=(I,J)}}\ar[u, "\delta'"],
    \end{tikzcd}\]
    where the vertical maps are those described in \Cref{thm:CremonaFisherSummary} for the data $(E_{\theta},F)$ and $(E,F)$ respectively.  The map $\phi_\theta$ is given by sending the equivalence class of a binary quartic form $g$ to that of $\theta^{-1} g$.
\end{lemma}
\begin{proof}
    Note firstly that the new map $\phi_\theta$ is well defined: scalar multiplication commutes with the action of $\PGL_2(F)$ on the forms, and the claimed invariants in the image are correct as $I(g)$ and $J(g)$ are homogeneous of degree $2$ and $3$ in the coefficients of the associated form $g$.

    Write $L:=L(E)=F[\epsilon]/\epsilon^3-3I\epsilon+J$ and $L_\theta:=L(E_{\theta})=F[\epsilon]/\epsilon^3-3\theta^2I\epsilon+\theta^3J$ for the \'etale algebras associated to the data of $(E,F)$ and $(E_{\theta},F)$ by \Cref{thm:CremonaFisherSummary}.  Moreover, let us write
        \begin{align*}
            \beta:L_\theta&\to L\\
            \epsilon&\mapsto \theta\epsilon.
    \end{align*}
    We begin by noting that the following diagram commutes:
    \begin{equation}
        \begin{tikzcd}
            L_\theta\ar[d, "\kappa"]\ar[r, "\beta"]&L\ar[d, "\kappa"]\\
            \bar{L}\ar[r, equal]&\bar{L}
        \end{tikzcd}
    \end{equation}
    where $\kappa$ are the Kummer maps induced by the inclusions $L, L_\theta\subset\bar{L}$ fixed in \Cref{thm:CremonaFisherSummary}.  Then, since the Weil pairing is preserved by the isomorphism $\varphi_\theta$, it is then clear from the definition of the maps $\alpha$ of \Cref{thm:CremonaFisherSummary} that the diagram below commutes:
    \[\begin{tikzcd}
        H^1(F, E_{\theta}[2])\ar[r, "\varphi_\theta^*"]\ar[d, "\alpha"]&H^1(F, E[2])\ar[d, "\alpha"]\\
        \ker(N_{L_\theta/F}:L_\theta^\times/L_\theta^{\times2}\to K^\times/K^{\times2})\ar[r, "\beta"]
        &\ker(N_{L/F}:L^\times/L^{\times2}\to K^\times/K^{\times2}),
    \end{tikzcd}\]
    where the vertical maps are those described in \Cref{thm:CremonaFisherSummary} for the data $(E_{\theta},F)$ and $(E,F)$ respectively.

    Using the commutativity of the diagram in \Cref{thm:CremonaFisherSummary} for the data $(E,F)$ and $(E_{\theta},F)$, and that the $\alpha$-maps are injective, we then see that the claim that $\delta'\circ \phi_\theta=\varphi_\theta^*\circ \delta'$ holds if and only if the diagram below commutes:
    \[\begin{tikzcd}
        \ker(N_{L_\theta/F}:L_\theta^\times/L_\theta^{\times2}\to K^\times/K^{\times2})\ar[r, "\beta"]
        &\ker(N_{L/F}:L^\times/L^{\times2}\to K^\times/K^{\times2})\\
        \set{\substack{\PGL_2(F)\textnormal{-equivalence classes of}\\\textnormal{binary quartic forms }g\in V_F\\\textnormal{with }(I(g),J(g))=(\theta^2I,\theta^3J)}}\ar[r, "\phi_\theta"]\ar[u, "z"]
        &\set{\substack{\PGL_2(F)\textnormal{-equivalence classes of}\\\textnormal{binary quartic forms }g\in V_F\\\textnormal{with }(I(g),J(g))=(I,J)}}\ar[u, "z"],
    \end{tikzcd}\]
    but this follows from the definition of the maps $z$:  if $g$ is a binary quartic form with invariants $(\theta^2I, \theta^3J)$ then (for an appropriate choice of $(x,y)\in F\times F$)
    \begin{align*}
        \beta\circ z(g)&=\beta\left(\frac{1}{3}(4\epsilon g(x,y)+g_4(g;x,y))\right)\\
        &=\frac{1}{3}(4\theta\epsilon g(x,y)+g_4(g;x,y))
        \\&\equiv\frac{1}{3}(4\theta^{-1}\epsilon g(x,y)+\theta^{-2}g_4(g;x,y))
        \\&=\frac{1}{3}(4\epsilon \theta^{-1}g(x,y)+g_4(\theta^{-1}g;x,y))
        \\&=z\circ\phi_\theta(g)
    \end{align*}
\end{proof}

Using \Cref{thm:BQF rep of twisted Kummer} we are then able to describe the local groups of the corestriction Selmer structure in terms of the correspondence to binary quartic forms from \Cref{cor:solBQFsAreE/2}.

\begin{corollary}\label{cor:local norms correspond to S-soluble BQFS}
    Let $F$ be a number field, and $K/F$ be a Galois extension.  Let $v\in\places_F$, assume that $w\in\places_K$ is a place extending $v$ such that $K_w/F_v$ is multiquadratic, and write $S:=\ker(F_v^\times/F_v^{\times2}\to K_w^\times/K_w^{\times2})$.  Let $E/F_v$ be an elliptic curve, with specified Weierstrass equation 
    \[E:y^2=x^3-\tfrac{I}{3}x-\tfrac{J}{27}.\]

    Using the notation of \Cref{thm:CremonaFisherSummary} for the data $(E,F_v)$ above, the map $\delta'$ restricts to a bijection between $\PGL_2(F)$-orbits of $(F,S)$-soluble binary quartic forms with invariants $I,J$ and elements of $\bigcap_{\theta\in S}\SSS_v^{(\theta)}(F;E)$.  Under this bijection, the identity element corresponds to the (unique!) $\PGL_2(F)$-orbit of binary quartic forms possessing a linear factor over $F$.

\end{corollary}
\begin{proof}
    By \Cref{thm:BQF rep of twisted Kummer} and \Cref{cor:solBQFsAreE/2} $\SSS_v^{(\theta)}(F;E)$ corresponds through $\delta'$ to the equivalence classes of binary quartic forms $g\in V_F$ with invariants $I,J$ such that $\theta g$ is $F$-soluble.  Thus the intersection corresponds precisely to the equivalence classes of $(F,S)$-soluble binary quartic forms.  That the second statement is equivalent is clear from \Cref{thm:CremonaFisherSummary}.
\end{proof}

We can now strengthen \Cref{thm:Integral BQF with proper invariants correspond to selmer elements} to identify the corestriction Selmer group as a subset of the equivalence classes of locally soluble binary quartic forms.

\begin{theorem}\label{thm:Corestriction Selmer corresponds to loc S-sol integral BQFs}
    Let $K/\QQ$ be a multiquadratic extension, and $S:=\ker(\QQ^\times/\QQ^{\times2}\to K^\times/K^{\times2})$.  Let $(A,B)\in\Epsilon$, and specify the Weierstrass equation for $E=E_{A,B}$ to be
    \[E:y^2=x^3+2^4Ax+2^6B.\]
    Then using the notation of \Cref{thm:CremonaFisherSummary} for the data $(E,\QQ)$: the map $\delta'$ induces a bijection between $\PGL_2(\QQ)$-equivalence classes of locally $S$-soluble binary quartic forms $g\in V_\ZZ$ with invariants $(I(g),J(g))=(2^4I(E),2^6J(E))$ and the intersection $\bigcap_{\theta\in S}\varphi_\theta^*\sel{2}(E_{\theta}/\QQ)\subseteq H^1(\QQ, E[2])$.

    Furthermore, the set of $g\in V_\ZZ$ having a linear factor (over $\QQ$) and invariants $(2^4I(E),2^6J(E))$ lie in a single $\PGL_2(\QQ)$-orbit, and this orbit maps to the identity element.
\end{theorem}
\begin{proof}
    To ease discussion below, for each $v\in\places_\QQ$ let $\delta'_v$ be the map from \Cref{thm:CremonaFisherSummary} with the data $(E, \QQ_v)$.  Moreover write $S_v$ for the image of $S$ in $\QQ_v^\times/\QQ_v^{\times2}$.

    Let $A$ be the set of $\PGL_2(\QQ)$-equivalence classes of locally soluble binary quartic forms $g\in V_\ZZ$ having invariants $(I(g),J(g))=(2^4I, 2^6J)$ such that: for every $v\in\places_\QQ$, the map $\delta'_v$ maps the equivalence class of $g$ to an element of $\bigcap_{\theta\in S}\SSS_v^{(\theta)}(F;E)$.  By \Cref{cor:local norms correspond to S-soluble BQFS}, $\bigcap_{\theta\in S}\SSS_v^{(\theta)}(F;E)$ corresponds under the $\delta'$ to the set of equivalence classes of $(\QQ_v,S_v)$-soluble binary quartic forms $g\in V_{\QQ_v}$ with invariants $(2^4I,2^6J)$.  Now by \Cref{thm:Integral BQF with proper invariants correspond to selmer elements}, $\delta'$ induces a bijection between $\bigcap_{\theta\in S}\varphi_\theta^*\sel{2}(E_{\theta}/\QQ)\subseteq\sel{2}(E/\QQ)$ and the set $A$, as required.
\end{proof}    

\Cref{thm:Corestriction Selmer corresponds to loc S-sol integral BQFs} tells us which $\PGL_2(\QQ)$-equivalence classes of integral binary quartic forms $g\in V_\ZZ$ to count in order to determine the average of $\bigcap_{\theta\in S}\varphi_\theta^*\sel{2}(E_{A,B}^{(\theta)}/\QQ)$ as $(A,B)\in\Epsilon$ varies.

\section{Recalling Bhargava--Shankar}
\label{sec:RecallingBS}
In this section we recall necessary counting results and definitions due to Bhargava--Shankar \cite{MR3272925}.
\subsection{Elliptic Curves and Families}
\label{subsec:ECandFamilies}
Firstly, we have the invariants corresponding to each elliptic curve.
\begin{definition}[\cite{MR3272925}*{\S3}]\label{def:naiveheightofEC}
    For each pair $(I,J)\in\RR^2$ we define the height to be
    \[H(I,J):=\max\set{\abs{I}^3, J^2/4}.\]
    For an elliptic curve $E/\QQ$, with associated invariants $I:=I(E)$ and $J:=J(E)$ (see \Cref{def:I(E) and J(E)}), we define the height of $E$ to be
    \[H'(E):=H(I,J).\]
    The discriminant of the pair $(I,J)$ is defined to be
    \[\Delta'(I,J):=\frac{4I^3-J^2}{27}\]
\end{definition}
\begin{rem}
    Note that if we present $E=E_{A,B}$ for $(A,B)\in\Epsilon$ then it is easy to see that the notion of height introduced above differs from the naive height by a constant factor:
    \[H'(E)=\max\set{\abs{-3A}^3, (-27B)^2}=\frac{27}{4}\max\set{4\abs{A}^3, 27B^2},\]
    so the ordering on elliptic curves is equivalent to that of the naive height.  Moreover,
    \[\Delta'(I,J)=-(4A^3+27B^2)=-\Delta(A,B)\]
    recovers the discriminant of the given model of $E$.
\end{rem}

We now introduce the `large families' of elliptic curves to which the results of \cite{MR3272925} apply.
\begin{definition}[\cite{MR3272925}*{\S3}]\label{def:cong cdn family}
    For each prime $p$, let 
    \[\Sigma_p\subset \set{(I,J)\in\ZZ_p^2~:~ \Delta'(I,J)\neq 0}\]
    be a nonempty closed subset with boundary of measure 0.  Moreover let $\Sigma_\infty$ be one of the following:
    \begin{align*}
    &\set{(I,J)\in\RR^2~:~\Delta'(I,J)<0},\ \set{(I,J)\in\RR^2~:~\Delta'(I,J)>0},
    \\&\set{(I,J)\in\RR^2~:~\Delta'(I,J)\neq0}.
    \end{align*}
    Associated to the data $\Sigma=(\Sigma_v)_{v\in\places_\QQ}$, we have a family of elliptic curves over $\QQ$, 
    \[\cF_{\Sigma}:=\set{E/\QQ~:~(I(E), J(E))\in \Sigma_v\quad \forall v\in\places_\QQ}.\]
    A family of elliptic curves $\cF$ is said to be defined by congruence conditions if $\cF=\cF_\Sigma$ for some $\Sigma=(\Sigma_v)_{v\in\places_\QQ}$ as above.
\end{definition}
Associated to a family of curves which is defined by congruence conditions, we have some additional data.
\begin{notation}\label{not:invariants for large family}
Let $\cF=\cF_\Sigma$ be a family of elliptic curves defined by congruence conditions.  Then we have
    \begin{itemize}
         \item $\Inv(\cF):=\set{(I(E), J(E))~:~E\in\cF}$.
         \item For each prime number $p$, $\Inv_p(\cF)$ is the set of $(I,J)$ in the $p$-adic closure of $\Inv(\cF)$ in $\ZZ_p^2$ for which $\Delta'(I,J)\neq 0$.
         \item $\Inv_\infty(\cF):=\Sigma_\infty$.
    \end{itemize}
\end{notation}
\begin{rem}
    It is not, in this generality, true that $\Inv_p(\cF_\Sigma)=\Sigma_p$.  Take, for instance, 
    \[\Sigma_p=\set{(I,J)\in\ZZ_p^2~:~I\equiv J\equiv 0\mod p\textnormal{ and } \Delta'(I,J)\neq 0}\]
    for every prime number $p$.  Then of course $\cF_\Sigma=\emptyset$, so $\Inv_p(\cF_\Sigma)=\emptyset$.
\end{rem}
The families that can be studied with the analytic tools of \cite{MR3272925} are those defined by congruence conditions that satisfy an additional `largeness' axiom.
\begin{definition}\label{def:large family}
    A family $\cF$ of elliptic curves which is defined by congruence conditions is further called a large family if for all but finitely many primes $p$ the set $\Inv_p(\cF)$ contains all pairs $(I,J)\in\ZZ_p^2$ such that $p^2\nmid \Delta(I,J)$.
\end{definition}

\begin{rem}
    Rephrased in terms of the associated elliptic curves, this definition states that: for sufficiently large $p$, $\Sigma_p$ contains all elliptic curves with reduction types $I_0$ and $I_1$.  
\end{rem}
\subsection{Counting Binary Quartic Forms}\label{subsec:CountingBQF}
By \Cref{thm:Integral BQF with proper invariants correspond to selmer elements}, finding the average of $\sel{2}(E/\QQ)$ as $E$ varies in a large family is equivalent to counting $\PGL_2(\QQ)$-equivalence classes of elements $g\in V_\ZZ$ with certain invariants.  Bhargava and Shankar count $\PGL_2(\ZZ)$-orbits in $V_\ZZ$, rather than $\PGL_2(\QQ)$-equivalence classes.  We postpone the relationship between the two counting problems to the next subsection, and simply present the results.
\begin{definition}\label{def:height of BQF}
The height of a binary quartic form $f\in V_\QQ$ with invariants $I,J$ is defined to be
\[H(f):=\max\set{\abs{I}^3, J^2/4}\]
\end{definition}
\begin{rem}
    The height of a binary quartic form is a function of its invariants.  In particular, the binary quartic forms $g\in V_\ZZ$ which have $\PGL_2(\QQ)$-equivalence classes corresponding (via \Cref{thm:Integral BQF with proper invariants correspond to selmer elements}) to $2$-Selmer group elements for $E/\QQ$ have height
    \[H(g)=2^{12}H'(E).\]
    Thus counting $\PGL_2(\QQ)$-equivalence classes of binary quartic forms of bounded height corresponds to counting the elements of all $E_{A,B}$ of bounded naive height. 
\end{rem}

In order to apply local conditions (such as solubility) to the equivalence classes of binary quartic forms, we require a notion of acceptable congruence conditions.
\begin{definition}[\cite{MR3272925}*{\S2.7}]\label{def:AcceptableFunction}
    A function $\psi:V_\ZZ\to [0,1]\subset \RR$ is said to be defined by congruence conditions if, for all primes $p$, there exist functions $\psi_p:V_{\ZZ_p}\to [0,1]$ satisfying
    \begin{itemize}
        \item[(i)] For all $f\in V_\ZZ$, the product $\prod_{p}\psi_p(f)$ converges to $\psi(f)$. 
        \item[(ii)] For each prime $p$, the function $\psi_p$ is locally constant outside some closed set $S_p\subset V_{\ZZ_p}$ of measure zero.
    \end{itemize}
    If additionally for all but finitely many primes $p$, we have $\psi_p(f)=1$ whenever $p^2\nmid \Delta(f)$, then we say $\psi$ is acceptable.
\end{definition}

We then have our notation for the relevant counts.
\begin{definition}
    For each real number $X>0$, we define $N(V_\ZZ^{(i)};X)$ to be the number of $\PGL_2(\ZZ)$-equivalence classes of irreducible elements $f\in V_\ZZ^{(i)}$ satisfying $H(f)<X$

    Let $\psi$ be an acceptable function with corresponding local functions $\psi_p$ which are $\PGL_2(\ZZ)$-invariant.  Then we further define $N_\psi(V_\ZZ^{(i)};X)$ to be the number of $\PGL_2(\ZZ)$-orbits of irreducible elements $f\in V_\ZZ^{(i)}$ satisfying $H(f)<X$, where each equivalence class is counted with weight $\psi(f)$.
\end{definition}

We then have their main counting machine.
\begin{theorem}[\cite{MR3272925}*{Theorem 2.21}]\label{BS:Theorem 2.21}
    Let $\psi:V_\ZZ\to[0,1]$ be an acceptable function defined by congruence conditions via local functions $\psi_p:V_{\ZZ_p}\to [0,1]$ which are $\PGL_2(\ZZ)$-invariant.  Then for $i\in\set{0,1,2+,2-}$
    \[N_{\psi}(V_{\ZZ}^{(i)};X)=N(V_{\ZZ}^{(i)};X)\prod_p\int_{f\in V_{\ZZ_p}}\psi_p(f)df+o(X^{5/6})\]
\end{theorem}

\subsection{Weighted integral orbits}\label{subsec:WeightedIntOrbits}
We now explain the reduction from $\PGL_2(\QQ)$-equivalence classes in $V_\ZZ$ to $\PGL_2(\ZZ)$-orbits.  Doing so makes use of a certain well-behaved weighting.  Ideally, in order to count $\PGL_2(\QQ)$-equivalence classes in $V_\ZZ$, one could simply count the number of $\PGL_2(\ZZ)$-orbits of $f\in V_\ZZ$ with weight $1/n(f)$ where $n(f)$ is the number of $\PGL_2(\ZZ)$-orbits inside the $\PGL_2(\QQ)$-equivalence class of $f$.  However, this weighting is not defined by congruence conditions and so \Cref{BS:Theorem 2.21} would not be possible.  In order to resolve this, one replaces $n(f)$ by a slightly different weight $m(f)$.

\begin{definition}[\cite{MR3272925}*{\S3.2}]\label{def:weights m(f)}
    For a binary quartic form $f\in V_{\ZZ}$ we define a weighting
        \[m(f):=\sum_{f'\in B(f)}\frac{\#\Aut_\QQ(f)}{\#\Aut_\ZZ(f')},\]
    where $B(f)$ denotes a set of representatives of orbits of the action of $\PGL_2(\ZZ)$ on the $\PGL_2(\QQ)$-equivalence class of $f\in V_\ZZ$, and $\Aut_R(f)$ is the stabiliser of $f$ in $\PGL_2(R)$.  Analogously there are local weights at each prime $p$ for $f\in V_{\ZZ_p}$
        \[m_p(f):=\sum_{f'\in B_p(f)}\frac{\#\Aut_{\QQ_p}(f)}{\#\Aut_{\ZZ_p}(f')},\]
    where $B_p(f)$ denotes a set of representatives of orbits for the action of $\PGL_2(\ZZ_p)$ on the $\PGL_2(\QQ_p)$-equivalence class of $f\in V_\ZZ$.
\end{definition}

This new weighting is defined by congruence conditions, which keeps us on track for using \Cref{BS:Theorem 2.21}.

\begin{proposition}[\cite{MR3272925}*{Prop 3.6}]\label{BS:prop3.6}
    Suppose $f\in V_\ZZ$ has invariants $I,J$ such that $\Delta'(I,J)\neq 0$, then $m(f)=\prod_p m_p(f)$.
\end{proposition}

Helpfully, weighting by $m(f)$ instead of $n(f)$ does not matter particularly for counting purposes -- they differ only in a density zero set.

\begin{lemma}[\cite{MR3272925}*{Lemma 2.4}, see also \cite{MR3272925}*{\S3.2}]\label{lem:m(f) is same as n(f) often enough} For $f\in V_\ZZ$, denote its $\PGL_2(\QQ)$-equivalence class in $V_\ZZ$ by $[f]_\QQ$ .  Then for sufficiently large $X>0$ and any $\epsilon>0$
    \[\#\set{[f]_\QQ~:~\substack{f\in V_\ZZ,\\H(f)<X\\\Delta'(I(f),J(f))\neq 0\\m(f)\neq n(f)}}\ll_\epsilon X^{3/4+\epsilon}\]
\end{lemma}

Thus, in order to count $\PGL_2(\QQ)$-equivalence classes of elements $f\in V_\ZZ$, it is enough to count $\PGL_2(\ZZ)$-equivalence classes of $f$ with weight $1/m(f)$, as the number of forms for which the correct weight differs from this is negligible when taking an average.

Finally, if we write $\psi_p:=\mathbbm{1}_p/m_p(f)$, where $\mathbbm{1}_p$ is the indicator function for the set of $\QQ_p$-soluble binary quartic forms then we see that $\psi=(\psi_p)_{p}$ an acceptable function defined by congruence conditions.
\begin{proposition}[\cite{MR3272925}*{Prop. 3.18}]\label{BS:Proposition 3.18}
    Let $p>2$ be an odd prime number.  If $f\in V_\ZZ$ is either not $\QQ_p$-soluble or $m_p(f)\neq 1$ then $p^2\mid \Delta'(I(f),J(f))$.
\end{proposition}

In particular we can use \Cref{BS:Theorem 2.21} to count our orbits of interest.  In doing so, it will be helpful to compute the $p$-adic masses on the right hand side of the equation there.  For this we have the following.
\begin{theorem}[\cite{MR3272925}*{Cor. 3.8}]\label{BS:Cor 3.8}
    Let $p$ be a prime and $\phi_p$ a continuous $\PGL_2(\QQ_p)$-invariant function on $V_{\ZZ_p}$, such that every element $f\in V_{\ZZ_p}$ in the support of $\phi_p$ has nonzero discriminant, is soluble and satisfies $2^4\cdot 3\mid I(f)$ and $2^6\cdot 3^3\mid J(f)$.  Then
    \begin{align*}
    &\int_{\ZZ_p}\frac{\phi_p(f)}{m_p(f)}df
    \\&=\abs{\frac{1}{27}}_p\Vol{\PGL_2(\ZZ_p)}\int_{\substack{(I,J)\in \ZZ_p^2\\\Delta(I,J)\neq 0}}\frac{1}{\#E^{I,J}(\QQ_p)[2]}\left(\sum_{\sigma\in E^{I,J}(\QQ_p)/2E^{I,J}(\QQ_p)}\phi_p(f_\sigma)\right)dIdJ
    \end{align*}
    where $f_\sigma$ is any element in $V_{\ZZ_p}$ corresponding to $\sigma$ under the correspondence in \Cref{thm:Integral BQF with proper invariants correspond to selmer elements}.
\end{theorem}

\section{Selmer Bundles}\label{sec:SelmerBundles}
In this section we generalise the work of Bhargava--Shankar \cite{MR3272925} to count average sizes of certain Selmer groups $\sel{\cL}(\QQ, E[2])$ as the pair $(E,\cL)$ varies.  We fix a large family $\cF$ of elliptic curves.
\subsection{\texorpdfstring{$2$}{2}-Selmer Bundles}
We shall firstly need a way to gather together the elements of general Selmer structures on elliptic curves in a large family $\cF$.
\begin{definition}\label{def:2selbundle}
    A $2$-Selmer bundle $\LLL$ (on $\cF$) is the data of, for each place $v\in\places_\QQ$, a subset
    \[\LLL_v(\cF)\subseteq \set{f\in V_{\QQ_v}~:~\substack{(2^{-4}I(f), 2^{-6}J(f))\in \Inv_v(\cF)\\\textnormal{and }f\textnormal{ is }\QQ_v\textnormal{-soluble}}},\]
    such that
    \begin{enumerate}[label=(\Roman*)]
        \item\label{enum:2selbundle infty set} $\LLL_\infty(\cF)$ is either the whole set or the subset of binary quartic forms which have a linear factor over $\RR$ -- we say that $\LLL_\infty(\cF)$ is type $2$ or type $1$ respectively for the two options;
        \item\label{enum:2selbundle p set} for each prime number $p$, $\LLL_p(\cF)\cap V_{\ZZ_p}$ is closed and open in $V_{\ZZ_p}$ and is a union of $\PGL_2(\QQ_p)$-equivalence classes;
        \item\label{enum:2selbundle unram cdn} for every $(I,J)\in \Inv(\cF)$, all but finitely many prime numbers $p$ satisfy 
        \[\set{f\in V_{\ZZ_p}~:~\substack{(I(f),J(f))=(2^4I, 2^6J)\\\textnormal{and }f\in \LLL_p(\cF)}}=\set{f\in V_{\ZZ_p}~:~\substack{(I(f), J(f))\in (2^4I, 2^6J)\\f\textnormal{ is }\QQ_p\textnormal{-soluble}}};\]
        \item \label{enum:2selbundle unif p axiom} there is a constant $C(\LLL)\in\RR_{>0}$ such that for every prime number $p\geq C(\LLL)$ and every $f\in V_\ZZ$: if $p^2\nmid \Delta(f)$ then $f\in \LLL_p(\cF)$.
    \end{enumerate}
    For each $2$-Selmer bundle $\LLL$ we denote the subset 
        \[\LLL(\cF):=\set{f\in V_\ZZ~:~\substack{
        (2^{-4}I(f),\ 2^{-6}J(f))\in \Inv(\cF)\\
        \forall v\in \places_\QQ,\ f\in \LLL_v(\cF)}}\subseteq V_\ZZ.
        \]
\end{definition}
\begin{rem}\label{rem:infinity cdn on 2selbundles}
    By definition, $\LLL(\cF)$ is automatically a union of $\PGL_2(\QQ)$-equivalence classes.  The restrictions on $\LLL_\infty(\cF)$ should be understood as follows.  If $\Delta'(I,J)<0$ then $\#E^{I,J}(\RR)/2E^{I,J}(\RR)=1$ so that being $\RR$-soluble is equivalent to having a linear factor by \Cref{cor:solBQFsAreE/2}.  If instead $\Delta'(I,J)>0$ then $\#E^{I,J}(\RR)/2E^{I,J}(\RR)=2$, and so a choice of subgroup can be either the whole group or trivial.  Under the correspondence \Cref{cor:solBQFsAreE/2} such a choice is equivalent to deciding whether to include all $\RR$-soluble forms with appropriate invariants or just the ones with a linear factor.  Our constraint essentially forces that this decision is uniform across all of the elliptic curves $E\in\cF$.
\end{rem}
    We can understand the set $V_\ZZ\cap\LLL_\infty(\cF)$ of binary quartic forms in terms of the $V_\ZZ^{(i)}$ of \Cref{def:V_Z^i}.  We list all of the possibilities in the table below for the readers convenience, for an explanation of this see the discussion at the start of \cite{MR3272925}*{\S2.1}.

\begin{table}[ht]\centering
    \begin{tabular}{|c|c|c|}
        \hline
        Type of $\LLL_\infty(\cF)$&$\Inv_\infty(\cF)$&$V_\ZZ\cap\LLL_\infty(\cF)$\\
        \hline
        $1$&$\set{(I,J)~:~\Delta(I,J)<0}$&$V_\ZZ^{(1)}$\\
        $1$&$\set{(I,J)~:~\Delta(I,J)>0}$&$V_\ZZ^{(0)}$\\
        $1$&$\set{(I,J)~:~\Delta(I,J)\neq0}$&$V_\ZZ^{(0)}\cup V_\ZZ^{(1)}$\\
        $2$&$\set{(I,J)~:~\Delta(I,J)<0}$&$V_\ZZ^{(1)}$\\
        $2$&$\set{(I,J)~:~\Delta(I,J)>0}$&$V_\ZZ^{(0)}\cup V_\ZZ^{(2+)}$\\
        $2$&$\set{(I,J)~:~\Delta(I,J)\neq0}$&$V_\ZZ^{(0)}\cup V_\ZZ^{(1)}\cup V_\ZZ^{(2+)}$\\
        \hline
    \end{tabular}
    \caption{The possibilities for $V_\ZZ\cap\LLL_\infty(\cF)$, dependent on the type of $\LLL_\infty(\cF)$ and $\Inv_\infty(\cF)$}\label{tab:Possibilities for Linfty(F)}
\end{table}

A $2$-Selmer bundle assigns to each elliptic curve a Selmer structure in a continuous manner.

\begin{definition}\label{def:selmer structures from selmer bundle}
    Let $\LLL$ be a $2$-Selmer bundle.  Then for each elliptic curve $E\in \cF$ and each place $v\in\places_\QQ$, writing $(I,J):=(I(E), J(E))$, we define the subset
    \[\LLL(E)_v:=\delta\left(\set{x\in E(\QQ_v)/2E(\QQ_v)~:~\substack{x\textnormal{ corresponds via \Cref{cor:solBQFsAreE/2}}\\\textnormal{(with Weierstrass model }E^{2^4I,2^6J}\textnormal{)}\\\textnormal{to a subset of }\LLL_v(\cF)}}\right)\subseteq H^1(\QQ_v, E[2]),\]
    where $\delta:E(\QQ_v)/2E(\QQ_v)\to H^1(\QQ_v, E[2])$ is the usual connecting map from the Galois cohomology of the multiplication-by-$2$ exact sequence on $E$.  

    We then define the associated Selmer set for each $E\in\cF$ to be
    \[\sel{\LLL}(\QQ, E[2]):=\set{x\in H^1(\QQ, E[2])~:~\res_v(x)\in \LLL(E)_v\ \forall v\in\places_\QQ},\]
    where $\res_v:H^1(\QQ, E[2])\to H^1(\QQ_v, E[2])$ is the usual restriction map.
\end{definition}

\begin{rem}\label{rem:infinity cdn on 2selbundles SECOND MENTION}
    Note that, as in \Cref{rem:infinity cdn on 2selbundles}, for each $2$-Selmer bundle $\LLL$ we have that $\LLL_\infty(\cF)$ has type $2$ if and only if for every $E\in\cF$ with $\Delta'(I(E),J(E))>0$, we have $\#\LLL(E^{I,J})_\infty=2$.
\end{rem}

\begin{lemma}
    Let $\LLL$ be a $2$-Selmer bundle, and $E\in \cF$ be an elliptic curve.  Then for all but finitely many places $v\in\places_\QQ$ we have
    \[\LLL(E)_v=H^1_{\nr}(\QQ_v,E[2])\]
\end{lemma}
\begin{proof}
    This follows from condition \Cref{enum:2selbundle unram cdn} of \Cref{def:2selbundle} .
\end{proof}

Of course there is a classical example of the construction above:  $2$-Selmer groups themselves!

\begin{example}\label{ex:2-Selmer group as 2-selstruct}
    Consider the $2$-Selmer bundle $\SSS$ given by, for each place $v\in\places_\QQ$, setting
    \[\SSS_v(\cF)=\set{f\in V_{\QQ_v}~:~\substack{(2^{-4}I(f), 2^{-6}J(f))\in \Inv_v(\cF)\\\textnormal{and }f\textnormal{ is }\QQ_v\textnormal{-soluble}}}.\]
    Indeed: Axioms \Cref{enum:2selbundle infty set}, \Cref{enum:2selbundle p set} and \Cref{enum:2selbundle unram cdn} are clear from the definitions (and also implicit in \cite{MR3272925}).  Finally Axiom \Cref{enum:2selbundle unif p axiom} follows from \Cref{BS:Proposition 3.18}.

    The associated Selmer group for $E\in\cF$ is the usual $2$-Selmer group, i.e.
    \[\sel{\SSS}(\QQ, E[2])=\sel{2}(E/\QQ).\]

    Similarly, for $\theta\in \QQ^\times$ we have the $2$-Selmer bundle $\SSS^{(\theta)}$ given by setting
    \[\SSS_v^{(\theta)}(\cF)=\set{f\in V_{\QQ_v}~:~\substack{(2^{-4}I(f), 2^{-6}J(f))\in \Inv_v(\cF)\\\textnormal{and }\theta f\textnormal{ is }\QQ_v\textnormal{-soluble}}}.\]
    By \Cref{thm:BQF rep of twisted Kummer}, the associated Selmer group for $E\in\cF$ is that of its twist by $\theta$, i.e.
    \[\sel{\SSS^{(\theta)}}(\QQ,E[2])=\varphi_\theta^*\sel{2}(E_{\theta}/\QQ),\]
    where $\varphi_\theta^*:H^1(\QQ, E_{\theta}[2])\cong H^1(\QQ, E[2])$ is the isomorphism induced by the $\QQ(\sqrt{\theta})$-isomorphism $E_{\theta}\to E$.
\end{example}

\subsection{Averages Sizes in Selmer Bundles}
\label{subsec:AvgSelmerBundles}
Here we use the statistical results of \cite{MR3272925}, as recalled in \Cref{sec:RecallingBS} to conclude a local mass formula for the average size of Selmer groups obtained by $2$-Selmer bundles.  This will follow from \Cref{BS:Theorem 2.21}, but we shall unpack the right hand side of this a little.  Firstly we establish the archimedean contribution.

\begin{proposition}\label{prop:Archimedean massL}
    If $\LLL$ is a $2$-Selmer bundle then
    \begin{align*}
    N(V_\ZZ\cap&\LLL_\infty(\cF);X)
    \\&=\frac{1}{27}\Vol{\PGL_2(\ZZ)\backslash\PGL_2(\RR)}\int_{\substack{(I,J)\in \Inv_\infty(\cF)\\H(I,J)<X}}\frac{\#\LLL(E^{I,J})_\infty}{\#E^{I,J}(\RR)[2]}dIdJ+O(X^{3/4+\epsilon})
    \end{align*}
\end{proposition}
\begin{proof}
    If $\#\LLL(E^{I,J})=2$ for $\Delta(I,J)>0$ then this is \cite{MR3272925}*{Final paragraph before Thm. 3.19}.  In general: note that by \cite{MR3272925}*{eq. (20)} and \cite{MR3272925}*{Prop. 2.8 and preceding discussion}, and in their notation, we have for every $i=0,1,2+,2-$, writing $n_1=2$ and $n_i=4$ for $i\neq 1$, then
    \begin{align*}
    &N(V_\ZZ^{(i)};X)
    \\&=\Vol{R_X(L^{(i)})}/n_i+O(X^{3/4+\epsilon})
    \\&=\begin{cases}
        \frac{1}{27}\Vol{\PGL_2(\ZZ)\backslash\PGL_2(\RR)}\int_{\substack{(I,J)\in\Inv_\infty(\cF)
        \\\Delta(I,J)>0\\H(I,J)<X}}\frac{1}{4}dIdJ+O(X^{3/4+\epsilon})& i=0,2+,2-\\
        \frac{1}{27}\Vol{\PGL_2(\ZZ)\backslash\PGL_2(\RR)}\int_{\substack{(I,J)\in\Inv_\infty(\cF)\\\Delta(I,J)<0\\H(I,J)<X}}\frac{1}{2}dIdJ+O(X^{3/4+\epsilon})& i=1
        \end{cases}
    \end{align*}
    The result then follows from \Cref{rem:infinity cdn on 2selbundles SECOND MENTION} and \Cref{tab:Possibilities for Linfty(F)}, checking each possibility for $\Inv_\infty(\cF)$ and type for $\LLL_\infty(\cF)$.
\end{proof}
We then we compute the local masses coming from the non archimedean places.
\begin{lemma}\label{lem:Weighted mass over L_p}
    Let $\LLL$ be a $2$-Selmer bundle.  Then for every prime $p$ we have
    \[\int_{\LLL_p(\cF)\cap V_{\ZZ_p}}\frac{1}{m_p(f)}df=\abs{\frac{2^{10}}{27}}_p\Vol{\PGL_2(\ZZ_p)}\int_{(I,J)\in \Inv_p(\cF)}\frac{\#\LLL(E^{I,J})_p}{\#E^{I,J}(\QQ_p)[2]}dIdJ\]
\end{lemma}
\begin{proof}
    Define the function $\phi_p:V_{\ZZ_p}\to \set{0,1}$, by $\phi_p(f)=1$ if and only if $f\in \LLL_p(\cF)$.  By definition, $\phi_p$ is continuous and so by \Cref{BS:Cor 3.8} we obtain
    \begin{align*}
    \int_{\LLL_p(\cF)}\frac{1}{m_p(f)}df&=
    \int_{\ZZ_p}\frac{\phi_p(f)}{m_p(f)}df
    \\&=\abs{\frac{1}{27}}_p\Vol{\PGL_2(\ZZ_p)}\int_{\substack{(I,J)\in \ZZ_p^2\\\Delta(I,J)\neq 0}}\frac{\#\LLL(E^{I,J})_p}{\#E^{I,J}(\QQ_p)[2]}dIdJ
    \\&=\abs{\frac{2^{10}}{27}}_p\Vol{\PGL_2(\ZZ_p)}\int_{\substack{(I,J)\in \Inv_p(\cF)}}\frac{\#\LLL(E^{I,J})_p}{\#E^{I,J}(\QQ_p)[2]}dIdJ
    \end{align*}
    where the final step is given by the variable change $(2^{4}I, 2^{6}J)\mapsto (I, J)$, noting that $E^{I,J}$ is $\QQ_p$ isomorphic to $E^{2^4I, 2^6J}$.
\end{proof}
We then have the weighted indicator function for counting elements in a $2$-Selmer bundle.
\begin{definition}
    For every $2$-Selmer bundle $\LLL$ and prime number $p$, let 
    \[\phi_{\LLL,p}:V_{\ZZ_p}\to\set{0,1}\]
    be the indicator function which is $1$ if $f\in \LLL_p(\cF)$ and $0$ else.  Moreover, let $\psi_{\LLL,p}:=\phi_{\LLL,p}/m_p$, and denote $\psi_{\LLL}:=\prod_{p}\psi_{\LLL,p}:V_{\ZZ}\to \set{0,1}$.
\end{definition}

Finally, we need to know that this weighted indicator function (with which we wish to apply \Cref{BS:Theorem 2.21}) is indeed acceptable in the sense of \Cref{def:AcceptableFunction}.
\begin{lemma}\label{lem:psiL is acceptable}
    Let $\LLL$ be a $2$-Selmer bundle.  Then the function $\psi_{\LLL}$ is acceptable (in the sense of \Cref{def:AcceptableFunction}), with local functions $\psi_{\LLL,p}=\phi_{\LLL,p}/m_p$.
\end{lemma}
\begin{proof}
    That these local functions converge to the required global one is immediate from \Cref{BS:prop3.6}.  The second condition, that $\psi_{\LLL,p}$ is locally constant outside of a closed set of measure $0$, is clear from the definitions of the maps $\phi_{\LLL,p}$ and $m_p$.  Finally, to obtain the final condition, we need that for sufficiently large prime $p$, $\psi_{\LLL, p}(f)=1$ whenever $p^2\nmid \Delta(f)$.  This follows from \Cref{BS:Proposition 3.18} and Axiom \Cref{enum:2selbundle unif p axiom} for the $2$-Selmer bundle $\LLL$.
\end{proof}

We are now ready to state the main statistical corollary.

\begin{theorem}\label{thm:AvgSizeOf2SelBundle}
Let $\LLL$ be a $2$-Selmer bundle.  Then, for each large family $\cF$ of elliptic curves,
\[\frac{\sum_{\substack{E\in \cF\\H'(E)<X}}(\#\sel{\LLL}(\QQ, E[2])-1)}{\sum_{\substack{E\in \cF\\H'(E)<X}}1}=2\MMM^\LLL_\infty(\cF;X)\prod_p \MMM^\LLL_p(\cF)+o(1),\]
where the local masses are
\begin{align*}
    \MMM^\LLL_\infty(\cF;X)&:=\frac{\int_{\substack{(I,J)\in \Inv_\infty(\cF)\\H(I,J)<X}}\frac{\#\LLL(E^{I,J})_\infty}{\#E^{I,J}(\RR)[2]}dIdJ}{\int_{\substack{(I,J)\in \Inv_\infty(\cF)\\H(I,J)<X}}dIdJ}
    \\\MMM^\LLL_p(\cF)&:=\frac{\int_{\substack{(I,J)\in \Inv_p(\cF)}}\frac{\#\LLL(E^{I,J})_p}{\#E^{I,J}(\QQ_p)[2]}dIdJ}{\int_{\substack{(I,J)\in \Inv_p(\cF)}}dIdJ}
\end{align*}
\end{theorem}
\begin{proof}
    By \Cref{lem:m(f) is same as n(f) often enough} (and the explanation below it) and \Cref{thm:Integral BQF with proper invariants correspond to selmer elements} the numerator of the left hand side is equal to the number of $\PGL_2(\ZZ)$-orbits in $\LLL(\cF)$ of height at most $2^{12}X$ with no rational linear factor counted with weights $1/m(f)$ (up to acceptable error).  Moreover, by \cite{MR3272925}*{Lemma 2.3} the number of $\PGL_2(\ZZ)$-orbits in $\LLL(\cF)$ of height at most $2^{12}X$ which are reducible but factor as a pair of irreducible quadratics is at worst $\bigo{X^{2/3+\epsilon}}$ and so by counting orbits of irreducible forms we will be within acceptable error of counting the orbits of forms with no rational linear factor. 

    Thus by \Cref{BS:Theorem 2.21}, \Cref{lem:psiL is acceptable} and \Cref{tab:Possibilities for Linfty(F)} we have
    \begin{align*}
        \sum_{\substack{E\in \cF\\H'(E)<X}}(\#\sel{\LLL}(\QQ, E[2])-1)
        &=N_{\psi_{\LLL}}(V_\ZZ\cap\LLL_\infty(\cF);2^{12}X)
        \\&=N(V_\ZZ\cap\LLL_\infty(\cF);2^{12}X)\prod_{p}\int_{f\in V_{\ZZ_p}}\psi_p(f)df+o(X^{5/6})
    \end{align*}
    Then, using the mass formulae in \Cref{prop:Archimedean massL} and \Cref{lem:Weighted mass over L_p}, this gives
    \begin{align*}
        &N(V_\ZZ\cap\LLL_\infty(\cF);2^{12}X)\prod_{p}\int_{f\in V_{\ZZ_p}}\psi_p(f)df+o(X^{5/6})
        \\&=\frac{2^{10}}{27}\Vol{\PGL_2(\ZZ)\backslash\PGL_2(\RR)}\int\limits_{\substack{(I,J)\in \Inv_\infty(\cF)\\H(I,J)<X}}\frac{\#\LLL(E^{I,J})_\infty}{\#E^{I,J}(\RR)[2]}dIdJ\\&\prod_p\left(\abs{\frac{2^{10}}{27}}_p\Vol{\PGL_2(\ZZ_p)}\int\limits_{(I,J)\in \Inv_p(\cF)}\frac{\#\LLL(E^{I,J})_p}{\#E^{I,J}(\QQ_p)[2]}dIdJ\right) + o(X^{5/6})
        \\&=2\left(\int\limits_{\substack{(I,J)\in \Inv_\infty(\cF)\\H(I,J)<X}}\frac{\#\LLL(E^{I,J})_\infty}{\#E^{I,J}(\RR)[2]}dIdJ\right)\prod_p\left(\int\limits_{(I,J)\in \Inv_p(\cF)}\frac{\#\LLL(E^{I,J})_p}{\#E^{I,J}(\QQ_p)[2]}dIdJ\right) + o(X^{5/6}).
    \end{align*}
    where the last equality uses that $\prod_p\Vol{\PGL_2(\ZZ_p)}=\zeta(2)^{-1}$ as well as the equality $\Vol{\PGL_2(\ZZ)\backslash\PGL_2(\RR)}=2\zeta(2)$.  Moreover by \cite{MR3272925}*{Thm 3.17}
    \[\sum_{\substack{E\in \cF\\H'(E)<X}}1=\left(\int_{\substack{(I,J)\in \Inv_\infty(\cF)\\H(I,J)<X}}dIdJ\right)\prod_p\left(\int_{(I,J)\in \Inv_p(\cF)}dIdJ\right) + o(X^{5/6}).\]
    Thus we have the result.
\end{proof}

This reduces computing the average size of certain Selmer structures on elliptic curves to proving that they can be packaged as Selmer groups arising from $2$-Selmer bundles, a purely algebraic task, and then applying \Cref{thm:AvgSizeOf2SelBundle} and computing local masses.

\subsection{Application: The Corestriction Selmer Bundle}
\label{subsec:CoresSelmerBundle}
Our goal is to apply \Cref{thm:AvgSizeOf2SelBundle} to count the average size of corestriction Selmer groups.  In order to do this, we have the, somewhat algebraic, task of showing that these Selmer groups arise from a $2$-Selmer bundle.  As in the previous section, we take $\cF$ to be a large family of elliptic curves.
\begin{definition}\label{def:cores bundle}
    Let $K/\QQ$ be a multiquadratic extension and let $S:=\ker(\QQ^\times/\QQ^{\times2}\to K^\times/K^{\times2})$.  Define the $2$-Selmer bundle $\CCC(K)$ given by, for each place $v\in\places_\QQ$,
        \[\CCC(K)_v(\cF)=\set{f\in V_{\QQ_v}~:~\substack{(2^{-4}I(f), 2^{-6}J(f))\in \Inv_v(\cF)\\\textnormal{and }f\textnormal{ is }(\QQ_v,S)\textnormal{-soluble}}}.\]
\end{definition}

Now we verify that this does indeed define a $2$-Selmer bundle.
\begin{lemma}\label{lem:CoresSelBundle is SelBundle}
    For every multiquadratic extension $K/\QQ$, $\CCC(K)$ is a $2$-Selmer bundle, and for each elliptic curve $E/\QQ$ and place $v\in\places_K$  the local group is that of the corestriction Selmer structure:
    \[\CCC(K)(E)_v=N_{K_w}E(K_w)/2E(\QQ_v).\]
\end{lemma}
\begin{proof}

    Note that for every $v\in\places_\QQ$ we can write
    \[\CCC(K)_v(\cF)=\bigcap_{\theta\in S}\SSS_v^{(\theta)}(\cF),\]
    where the $\SSS_v^{(\theta)}(\cF)$ are as in \Cref{ex:2-Selmer group as 2-selstruct}.  It is clear that each $\SSS_v^{(\theta)}$ satisfies both Axioms \Cref{enum:2selbundle infty set} and \Cref{enum:2selbundle p set}, and so since these axioms are invariant under taking finite intersections they hold also for $\CCC(K)_v(\cF)$.

    By \Cref{cor:local norms correspond to S-soluble BQFS} we see that for each pair $(I,J)\in\Inv(\cF)$, and place $v\in\places_\QQ$ the set of $\PGL_2(\QQ_v)$-equivalence classes in
    \[\set{f\in \CCC(K)_v(\cF)~:~(I(f),J(f))=(I,J)}\]
    corresponds to the local corestriction group $\CCC_v(K/\QQ;E^{I,J})\subset H^1(\QQ, E^{I,J}[2])$.   By definition, all but finitely many $v$ satisfy 
    \[\CCC_v(K/\QQ;E^{I,J})=H^1_{\nr}(\QQ, E[2])=\SSS_v(\QQ; E^{I,J}).\]

    Now for all but finitely many $v$ the local corestriction group corresponds to the set of soluble forms with invariants $I,J$ as required for Axiom \Cref{enum:2selbundle unram cdn}.

    It remains to prove that Axiom \Cref{enum:2selbundle unif p axiom} holds.  Let $p\geq 5$ be a prime number which is coprime to $\Delta_K$.  Let $f\in V_\ZZ$ be such that $p^2\nmid \Delta(f)$, with associate invariants $(I,J)=(I(f),J(f))$.  By \Cref{BS:Proposition 3.18} $f$ is $\QQ_p$-soluble and so its $\PGL_2(\QQ_p)$-equivalence class corresponds to an element of $\SSS_p(\QQ;E)$.

    Note that (by Tate's algorithm) the elliptic curve $E=E^{I,J}/\QQ_p$ has reduction type $I_0$ or $I_1$ by the assumption on $\Delta(f)$.  By \Cref{for good I0I1 unram have cores is sel} we have
    \[\CCC_p(K/\QQ;E)=\SSS_p(F;E),\]
    and so the $\PGL_2(\QQ_p)$-equivalence class of $f$ corresponds to an element of $\CCC_p(K/\QQ;E)$, which by \Cref{thm:Corestriction Selmer corresponds to loc S-sol integral BQFs} implies $f\in \CCC(K)_v(\cF)$.
\end{proof}

Thus we have the following corollary of \Cref{thm:AvgSizeOf2SelBundle}.

\begin{corollary}\label{cor:large family cores selmer avg}
    Let $K/\QQ$ be a multiquadratic extension.  Then,
        \[\frac{\sum_{\substack{E\in \cF\\H'(E)<X}}(\#\sel{\CCC(K)}(\QQ, E[2])-1)}{\sum_{\substack{E\in \cF\\H'(E)<X}}1}=2\MMM^{\CCC(K)}_\infty(\cF;X)\prod_p \MMM^{\CCC(K)}_p(\cF)+o(1),\]
    where the local masses are
    \begin{align*}
        \MMM^{\CCC(K)}_\infty(\cF;X)&:=\frac{\int_{\substack{(I,J)\in \Inv_\infty(\cF)\\H(I,J)<X}}\frac{\#\CCC(K)(E^{I,J})_\infty}{\#E^{I,J}(\RR)[2]}dIdJ}{\int_{\substack{(I,J)\in \Inv_\infty(\cF)\\H(I,J)<X}}dIdJ},
        \\\MMM^{\CCC(K)}_p(\cF)&:=\frac{\int_{\substack{(I,J)\in \Inv_p(\cF)}}\frac{\#\CCC(K)(E^{I,J})_p}{\#E^{I,J}(\QQ_p)[2]}dIdJ}{\int_{\substack{(I,J)\in \Inv_p(\cF)}}dIdJ}.
    \end{align*}
\end{corollary}
\begin{proof}
    Immediate from \Cref{lem:CoresSelBundle is SelBundle} and \Cref{thm:AvgSizeOf2SelBundle}.
\end{proof}

\noindent Of course: the local masses above depend on the family $\cF$, and so we cannot really go much further in this generality.
\section{Average Size of Intersections}\label{sec:LocalDensities}
We will now apply \Cref{cor:large family cores selmer avg} to the most common family of interest:  all elliptic curves.  We will begin by stating the main result, which will require a definition, and the rest of the section will comprise a proof of this and then of the theorems in the introduction.

\begin{definition}\label{def:L_v}
    For every multiquadratic extension $K/\QQ$ and each prime number $p\geq 5$ define local factors
    \[L_p(\CCC(K)):=\begin{cases}
        \frac{(p-1)(p^4-p^3+p^2-p+1)(46p^5+62p^4+79p^3+84p^2+84p+48)}{48(p^{10}-1)} & \substack{\textnormal{if }K/\QQ\textnormal{ is ramified and}\\\textnormal{quadratic at }p,}\\
        \frac{16p^{11}+16p^{10}-8p^9+8p^8-8p^7-10p^6-4p^5+7p^4-p^3-8p^2-24p-1}{16(p^{10}-1)(p+1)} & \substack{\textnormal{if }K/\QQ\textnormal{ is unramified and}\\\textnormal{quadratic at }p,}\\
        \frac{(p-1)(p^4-p^3+p^2-p+1)(5p^5+15p^4+13p^3+9p^2+13p+8)}{8(p^{10}-1)} &\substack{\textnormal{if }K/\QQ\textnormal{ is biquadratic at }p,}\\
        1&\substack{\textnormal{if }K/\QQ\textnormal{ is totally split at }p.}
    \end{cases}\]
    For $p\in\set{2,3}$ we define some `coarse' local factors
    \[L_p(\CCC(K)):=\begin{cases}
        1&\textnormal{if }K/\QQ\textnormal{ is totally split at }p,\\
        \frac{1}{2^{2+[K_w:\QQ_2]}}&\textnormal{if }p=2\textnormal{ and }K/\QQ\textnormal{ is not totally split at }p,\\
        \frac{1}{4}&\textnormal{if }p=3\textnormal{ and }K/\QQ\textnormal{ is not totally split at }p.
    \end{cases}\]
    Moreover, define an archimedean factor
    \[L_{\infty}(\CCC(K)):=\begin{cases}
        \frac{1}{2}&\textnormal{if }K\textnormal{ is real,}\\
        \frac{9}{20}&\textnormal{if }K\textnormal{ is imaginary.}
    \end{cases}\]
\end{definition}

These local factors allow us to, concisely describe the average size of corestriction Selmer groups.
\begin{theorem}\label{thm:Average Size of Cores Selmer FINAL for MQ}
    Let $K/\QQ$ be a multiquadratic extension, and for concision write
    \[A(K):=\lim_{X\to\infty}\frac{\sum_{\substack{E\in \Epsilon\\H'(E)<X}}\left(\#\sel{\CCC(K)}(\QQ, E[2])-1\right)}{\sum_{\substack{E\in \Epsilon\\H'(E)<X}}1}.\]
    Then we have inequalities
    \[4\prod_{v\in\places_\QQ} L_v(\CCC(K))\leq A(K)\leq 4\prod_{\substack{v\in\places_\QQ\\v\nmid 6}} L_v(\CCC(K)).\]
    In particular, if $2$ and $3$ are totally split in $K/\QQ$ we have an equality
    \[A(K)=4\prod_{v\in\places_\QQ} L_v(\CCC(K)).\]
\end{theorem}

\subsection{The \texorpdfstring{$p$}{p}-adic Masses}
We begin by computing each of the $p$-adic masses in \Cref{cor:large family cores selmer avg} for the family of all elliptic curves.
\begin{notation}\label{not:all EC}
    For each prime number $p$, let
    \[\Epsilon_p:=\set{(A,B)\in\ZZ_p^2~:~\substack{\bullet 4A^3+27B^2\neq0\\\bullet (A,B)\not\in p^4\ZZ_p\times p^6\ZZ_p}}.\]
    Moreover, for $(A,B)\in\Epsilon_p$ we define $\Delta(A,B):=4A^3+27B^2$.
\end{notation}
\begin{lemma}\label{lem:MQ ext local mass as nmidx integral}
    Let $K/\QQ$ be a multiquadratic extension.  For every prime number $p$,
    \[\MMM_p^{\CCC(K)}(\Epsilon)=\abs{\frac{1}{2}}_p\frac{\int_{(A,B)\in\Epsilon_p}\#\left(\frac{E_{A,B}(\QQ_p)}{N_{K_w/\QQ_p}E_{A,B}(K_w)+2E_{A,B}(\QQ_p)}\right)^{-1}dAdB}{1-p^{-10}},\]
    where $w\in\places_K$ is a place extending the place at $p$.
\end{lemma}
\begin{proof}
    We begin by analysing the integrand in the numerator of
    \[\MMM^{\CCC(K)}_p(\cF):=\frac{\int_{\substack{(I,J)\in \Inv_p(\cF)}}\frac{\#\CCC(K)(E^{I,J})_p}{\#E^{I,J}(\QQ_p)[2]}dIdJ}{\int_{\substack{(I,J)\in \Inv_p(\cF)}}dIdJ}.\]
    
    Let $E=E^{I,J}$, then since $E(\QQ_p)$ has a finite index subgroup isomorphic to the additive group $\ZZ_p$ (see e.g. \cite{silverman2009arithmetic}*{VII Prop. 6.3}), we have
    \[\frac{\#E(\QQ_p)/2E(\QQ_p)}{\#E(\QQ_p)[2]}=\frac{\#\ZZ_p/2\ZZ_p}{\#\ZZ_p[2]}=\abs{\frac{1}{2}}_p.\]
    Moreover, writing $w$ for a place of $K$ extending $p$,
    \[\#\CCC(K)(E)_p=\#(N_{K_w/\QQ_p}E(K_w)/2E(\QQ_p))=\frac{\#E(\QQ_p)/2E(\QQ_p)}{\#E(\QQ_p)/\braces{N_{K_w/\QQ_p}E(K_w)+2E(\QQ_p)}}.\]
    Therefore, using \Cref{lem:CoresSelBundle is SelBundle}, by definition the numerator of $\MMM_p^{\CCC(K)}(\Epsilon)$ is the integral (as $E=E^{I,J}$ varies) of the expression
    \[\frac{\#\CCC(K)(E)_p}{\#E(\QQ_p)[2]}=\abs{\frac{1}{2}}_p\frac{1}{\#E(\QQ_p)/\braces{N_{K_w/\QQ_p}E(K_w)+2E(\QQ_p)}}\]
    Also, for every prime number $p\neq 3$ we have
    \[\Inv_p(\Epsilon)=\set{(-3A,-27B)\in\ZZ_p^{2}~:~v_p(A)<4\textnormal{ or }v_p(B)<6}=\Epsilon_p,\]
    Therefore, by the above we have (for $p\neq 3$)
    \[\MMM_p^{\CCC(K)}(\Epsilon)=\abs{\frac{1}{2}}_p\frac{\int_{(A,B)\in\Epsilon_p}\#\left(\frac{E_{A,B}(\QQ_p)}{N_{K_w/\QQ_p}E_{A,B}(K_w)+2E_{A,B}(\QQ_p)}\right)^{-1}dAdB}{1-p^{-10}},\]
    as required.  When $p=3$, the change of variables $(I,J)\mapsto (A,B)$ contributes a factor of $\abs{81}_3$ to both the numerator and denominator, so that the above expression holds for $p=3$ also.
\end{proof}
It remains to compute the integral in the numerator of $\MMM_p^{\CCC(K)}(\Epsilon)$.  We recall some useful local density calculations.

\begin{table}
    \begin{tabular}{|c|c|c|c|}
    \hline
    &\multicolumn{3}{c|}{$F$}\\
    \cline{2-4}
    $\mu_p(Z_{F}(i))$&\multicolumn{2}{c|}{Quadratic}&Biquadratic\\
    \cline{2-3}
    &Unramified&Ramified&\\
    \hline
    $i=1$&$\frac{(p-1)(p^7+p^6+p^5+p^3+p+1)}{p^9(p+1)}$&$\frac{(p-1)(p^5+1)(p^3+p^2+1)}{2p^9}$&$\frac{(p^5+1)(p-1)(p^2+3p+1)}{2p^7(p+1)}$\\
    $i=2$&$\frac{(p-1)\braces{p^2-p+1}}{6p^7(p+1)}$&$\frac{(p-1)(p^5+1)(2p^2-2p-1)}{12p^7(p+1)}$&$\frac{(p^5+1)(p-1)\braces{p^4-p^3+p^2+6p+6}}{6p^9(p+1)}$\\
    $i>2$&$0$&$0$&$0$\\
    \hline
    \end{tabular}
    \caption{For $p\geq 5$ and multiquadratic $F/\QQ_p$, the $p$-adic densities of the sets where $\phi_F=i$}\label{tab:MQDENS}
\end{table}

\begin{proposition}[\cite{PatMQG}*{Proposition 5.7}]\label{prop:local densities with const norm index}
    Let $p\geq 5$ be a prime number and $F/\QQ_p$ be a multiquadratic extension, and for brevity for each $i\geq 0$ write
    \[Z_F(i):=\set{(A,B)\in\Epsilon_p~:~\dim_{\FF_2}\frac{E_{A,B}(\QQ_p)}{N_{F/\QQ_p}E_{A,B}(F)+2E_{A,B}(\QQ_p)}=i}.\]
    Then the $p$-adic densities $\mu_p(Z_F(i))$ are given by \Cref{tab:MQDENS}.
\end{proposition}

\begin{lemma}\label{lem:p-adic mass for cores and all EC}
    Let $K/\QQ$ be a multiquadratic extension, and $p\geq 5$ be a prime number.  Then
    \begin{align*}
        \frac{\int_{(A,B)\in\Epsilon_p}\#\left(\frac{E_{A,B}(\QQ_p)}{N_{K_w/\QQ_p}E_{A,B}(K_w)+2E_{A,B}(\QQ_p)}\right)^{-1}dAdB}{1-p^{-10}}
        &=L_p(\CCC(K)).
    \end{align*}
\end{lemma}
\begin{proof}
    In the language of \Cref{prop:local densities with const norm index}, we have
    \[\int_{(A,B)\in\Epsilon_p}\#\left(\frac{E_{A,B}(\QQ_p)}{N_{K_w/\QQ_p}E_{A,B}(K_w)+2E_{A,B}(\QQ_p)}\right)^{-1}dAdB=\sum_{i=0}^{2}2^{-i}\mu_p(Z_{K_w}(i)),\]
    Since $\mu_p(Z_{K_w}(0))=1-\mu_p(Z_{K_w}(1))-\mu_p(Z_{K_w}(2))$, the result is a calculation in each case.
\end{proof}

\subsection{The Archimedean Contribution}
We now compute the archimedean factor in \Cref{cor:large family cores selmer avg} for the family $\Epsilon$.

\begin{lemma}\label{lem:arch C(K)(E)/E(R)[2]}
    Let $K/\QQ$ be a multiquadratic field, and $(I,J)\in\RR^2$ be elements such that $\Delta'(I,J)\neq 0$.  Then
    \[\frac{\#\CCC(K)(E^{I,J})_\infty}{\#E^{I,J}(\RR)[2]}=\begin{cases}
        \frac{1}{4}&\textnormal{if }K\textnormal{ is imaginary and }\Delta'(I,J)>0\\
        \frac{1}{2}&\textnormal{else}
    \end{cases}\]
\end{lemma}
\begin{proof}
    Let $w\in\places_K$ be an archimedean place, and denote $E:=E^{I,J}$.  Then note that by \Cref{lem:CoresSelBundle is SelBundle}
    \[\CCC(K)(E)_\infty=N_{K_w/\RR}E(K_w)/2E(\RR).\]
    The case that $K$ is real (so $K_w=\RR$) is then clear (see e.g. \cite{MR457453}*{Prop 3.7}).

    If, on the other hand, $K$ is imaginary then noting that $N_{\CC/\RR}E(\CC)=2E(\RR)$ we have
    \[\frac{\#(N_{\CC/\RR}E(\CC)/2E(\RR))}{\#E(\RR)[2]}=\begin{cases}
        \frac{1}{2}&\textnormal{if }\Delta'(I,J)<0\\
        \frac{1}{4}&\textnormal{if }\Delta'(I,J)>0,
    \end{cases}\]
    as required, since $\Delta'(I,J)$ is the discriminant of the elliptic curve $E^{I,J}$.
\end{proof}

\begin{lemma}\label{lem:arch contrib for cores selmer avg}
    Let $K/\QQ$ be a multiquadratic extension.  Then we have an equality
    \[\MMM^{\CCC(K)}_\infty(\Epsilon;X)=\begin{cases}
        \frac{1}{2}&\textnormal{if }K\textnormal{ is real,}\\
        \frac{9}{20}&\textnormal{if }K\textnormal{ is imaginary.}
    \end{cases}\]

\end{lemma}
\begin{proof}
    If $K$ is real then this is immediate from the definition of $\MMM^{\CCC(K)}_\infty(\Epsilon;X)$ and \Cref{lem:arch C(K)(E)/E(R)[2]}.  If $K$ is imaginary then by \Cref{lem:arch C(K)(E)/E(R)[2]} we have
    \begin{align*}
        \MMM^{\CCC(K)}_\infty(\cF;X)
        &=\frac{\int_{\substack{(I,J)\in \RR^2\\\Delta'(I,J)\neq 0\\H(I,J)<X}}\frac{\#\CCC(K)(E^{I,J})_\infty}{\#E^{I,J}(\RR)[2]}dIdJ}{\int_{\substack{(I,J)\in \RR^2\\\Delta'(I,J)\neq 0\\H(I,J)<X}}dIdJ}
        \\&=\frac{1}{2}-\frac{\int_{\substack{(I,J)\in \RR^2\\H(I,J)<X\\\Delta'(I,J)>0}}dIdJ}{4\int_{\substack{(I,J)\in \RR^2\\\Delta'(I,J)\neq 0\\H(I,J)<X}}dIdJ}
        \\&=\frac{9}{20}.
    \end{align*}
    where the final equality is elementary calculus.
\end{proof}

We now combine the results above to prove \Cref{thm:Average Size of Cores Selmer FINAL for MQ}.

\begin{proof}[Proof of \Cref{thm:Average Size of Cores Selmer FINAL for MQ}]
    By \Cref{cor:large family cores selmer avg}, computing all of the masses except those at $2$ and $3$ using \Cref{lem:p-adic mass for cores and all EC}, and \ref{lem:arch contrib for cores selmer avg}, we have
    \begin{equation}\label{eq:A(K) in cores final thm}
    A(K)=\braces{\prod_{p\in\set{2,3}}\MMM_p^{\CCC(K)}(\Epsilon)}2L_\infty(\CCC(K))\prod_{\substack{p\geq 5\\\textnormal{prime}}} L_p(\CCC(K)).
    \end{equation}
    For $p\in\set{2,3}$, using \cite{Paterson2021}*{Lemma 5.3} and the fact that if $K/\QQ$ is totally split at $p$ then the local norm is the identity map, we know that for every elliptic curve $E/\QQ_p$
    \[\frac{1}{\#\braces{E(\QQ_p)/(N_{K_w/\QQ_p}E(K_w)+2E(\QQ_p))}}\leq 1.\]
    Combining these bounds with \Cref{lem:MQ ext local mass as nmidx integral} we obtain
    \[\abs{\frac{1}{2}}_p\cdot L_p(\CCC(K))\leq \MMM_p^{\CCC(K)}(\Epsilon)\leq \abs{\frac{1}{2}}_p.\]
    Combining with the identity \Cref{eq:A(K) in cores final thm} we obtain the claimed result.
\end{proof}

\begin{rem}
    It is clear that the `coarse' local factors $L_2$ and $L_3$ are approximations of the correct factors for $\MMM_2^{\CCC(K)}$ and $\MMM_3^{\CCC(K)}$.  Broadly, the main ingredient going into computing the local factors $L_p$ for $p\geq 5$ is the studious account of local norm indices at $p$ in \cite{PatMQG}, which itself relies on performing Tate's algorithm carefully.  Certainly, for $p\in\set{2,3}$ one could go through Tate's algorithm, compute the corresponding results, and then find the correct factors $L_2$ and $L_3$.
\end{rem}

We state, as corollary, the bounds that this result leaves us with for the average dimension of corestriction Selmer groups.
\begin{corollary}\label{cor:Average Size of Cores Selmer for nat ordering and MQ}
    Let $K/\QQ$ be a multiquadratic extension, then
    \[\limsup_{X\to\infty}\frac{\sum_{(A,B)\in\Epsilon(X)}\dim\sel{\CCC(K)}(\QQ, E[2])}{\#\Epsilon(X)}\leq 4\prod_{\substack{v\in\places_\QQ\\v\nmid6}}L_v(\CCC(K)).\]
\end{corollary}
\begin{proof}
    This follows from \Cref{thm:Average Size of Cores Selmer FINAL for MQ} via the inequality $r\leq 2^{r}-1$.
\end{proof}

\subsection{Proofs of Main Theorems}
We now prove the results from the introduction.
\begin{theorem}[\Cref{thm:INTRO average size of intersection is product of densities}]\label{thm:average size of intersection is product of densities}
    Let $S$ be a set of squarefree integers.  Then,
    \[\lim_{X\to\infty}\frac{\sum_{E\in\Epsilon(X)}\#\bigcap\limits_{D\in S}\sel{2}(E_D/\QQ)}{\#\Epsilon(X)}=1+2\prod_{v\in\places_\QQ}\delta_{S,v}>1,\]
    where the $\delta_{S,v}:=\MMM_{v}^{\CCC(K)}(\Epsilon;X)$ are local densities defined in \Cref{cor:large family cores selmer avg} depending on $S$.
\end{theorem}
\begin{proof}
    By \Cref{prop:Multiquadratic extn desc of CandF with Twisted Kummers}, if $K/\QQ$ is a multiquadratic extension and $S=\ker(\QQ^\times/\QQ^{\times 2}\to K^\times/K^{\times 2})$ then
    \[\sel{\CCC(K)}(\QQ,E[2])=\bigcap_{\theta\in S}\sel{2}(E_\theta/\QQ).\]
    The result now immediate from \Cref{cor:large family cores selmer avg}.
\end{proof}

\begin{theorem}[\Cref{thm:INTRO prob nonzero goes to zero}]\label{thm:prob nonzero goes to zero}
    Let $D$ be a squarefree integer, then
    \[\limsup_{X\to\infty}\frac{\#\set{E\in\Epsilon(X)~:~\sel{2}(E/\QQ)\cap\sel{2}(E_D/\QQ)\neq 0}}{\#\Epsilon(X)}\ll \braces{\frac{23}{24}}^{\omega(D)},\]
    where the implied constant is independent of $D$.
\end{theorem}
\begin{proof}
For each squarefree integer $D\neq 1$, we have (again using \Cref{prop:Multiquadratic extn desc of CandF with Twisted Kummers} as above)
\[{\#\set{E\in\Epsilon(X)~:~\sel{2}(E/\QQ)\cap\sel{2}(E_D/\QQ)\neq 0}}\leq \sum_{E\in\Epsilon(X)}\dim\sel{\CCC(\QQ(\sqrt{D}))}(\QQ,E[2]),\]
this follows from \Cref{cor:Average Size of Cores Selmer for nat ordering and MQ} after observing that the ramified primes provide the bound $\prod_{v\in\places_\QQ}L_v(\CCC(K))\ll \braces{\frac{23}{24}}^{\omega(D)}$ as required.
\end{proof}

\begin{theorem}[\Cref{thm:INTROTHM imprecise on decomp}]\label{thm:THM imprecise on decomp}
    For each fixed squarefree integer $D$,
    \[\limsup_{X\to\infty}\frac{\#\set{E\in\Epsilon(X)~:~\eta_{E,D}\textnormal{ is not an isomorphism}}}{\#\Epsilon(X)}\ll \braces{\frac{23}{24}}^{\omega(D)},\]
    where the implied constant is independent of $D$.
\end{theorem}
\begin{proof}
    By \cite{MR4400944}*{Lemma 6.10}, if $E(\QQ)[2]=0$ and $\eta_{E,D}$ is not an isomorphism then $\sel{\CCC(\QQ(\sqrt{D}))}(\QQ, E[2])\neq0$.  In particular, since Hilbert's irreducibility theorem implies that the proportion of $E\in\Epsilon(X)$ for which $E(\QQ)[2]\neq0$ is $o(1)$, we need only bound the proportion of $E\in\Epsilon(X)$ for which $\sel{\CCC(\QQ(\sqrt{D}))}(\QQ, E[2])\neq0$.  Thus the claim follows from \Cref{cor:Average Size of Cores Selmer for nat ordering and MQ}, after observing that the ramified primes provide the bound $\prod_{v\in\places_\QQ}L_v(K)\ll \frac{23}{24}^{\omega(D)}$ as required.
\end{proof}

\begin{corollary}[\Cref{thm:INTRO prob nontrivial is positive}]\label{thm:prob nontrivial is positive} Let $D$ be a squarefree integer.  Then,
    \[\liminf_{X\to\infty}\frac{\#\set{E\in\Epsilon(X)~:~\sel{2}(E/\QQ)\cap \sel{2}(E_D/\QQ)\neq 0}}{\#\Epsilon(X)}>0.\]
\end{corollary}
\begin{proof}
Let $K=\QQ(\sqrt{D})$, so that by \Cref{prop:Multiquadratic extn desc of CandF with Twisted Kummers} we may replace the intersection with $\sel{\CCC(K)}(\QQ,E[2])$ for brevity.  Denote
\[\cA(X):=\set{E\in\Epsilon(X)~:~\sel{\CCC(K)}(\QQ,E[2])\neq 0}.\]
We apply the Cauchy--Schwartz inequality to obtain for each $X>0$
\begin{align*}
    \braces{\sum_{E\in\Epsilon(X)}\braces{\#\sel{\CCC(K)}(\QQ,E[2])-1}}^2
    &\leq \#\cA(X)\braces{\sum_{E\in\Epsilon(X)}\braces{\#\sel{\CCC(K)}(\QQ,E[2])-1}^2}.
\end{align*}
Rearranging and dividing through by $\#\Epsilon(X)$, we have for large enough $X$
\begin{align*}
    \frac{\#\cA(X)}{\#\Epsilon(X)}&\geq
    \braces{\frac{\sum_{E\in\Epsilon(X)}\braces{\#\sel{\CCC(K)}(\QQ,E[2])-1}}{\#\Epsilon(X)}}^2\braces{\frac{\sum_{E\in\Epsilon(X)}\braces{\#\sel{\CCC(K)}(\QQ,E[2])-1}^2}{\#\Epsilon(X)}}^{-1}.
\end{align*}
Applying \Cref{thm:average size of intersection is product of densities}, and bounding the second moment using \cite{swaminathan2021second}*{Theorem 1.1}, we obtain
\begin{align*}
\liminf_{X\to\infty}\frac{\#\cA(X)}{\#\Epsilon(X)}
&\geq \braces{\prod_{v\in\places_\QQ}\delta_{\set{1,D},v}}^2\liminf_{X\to\infty}\braces{\frac{\sum_{E\in\Epsilon(X)}\braces{\#\sel{2}(E/\QQ)-1}^2}{\#\Epsilon(X)}}^{-1}\\
&\geq \frac{\braces{\prod_{v\in\places_\QQ}\delta_{\set{1,D},v}}^2}{15}.
\end{align*}
\end{proof}

\section{Heuristic for Intersections}\label{sec:Heuristic}
We now present a heuristic to explain the statistical behaviour of $\sel{2}(E/\QQ)\cap \sel{2}(E_D/\QQ)$, which we conjecture to hold for certain families of $E/\QQ$ and is already known to hold in some cases.  In \Cref{subsec:NaiveModel} we present a na\"ive model for the behaviour of this intersection, postponing the proofs of necessary lemmata to the end of the subsection.  Following on, in \Cref{subsec:MainHeuristic} we discuss when this model should be correct, and when it should be `almost' correct in a formal sense.  Finally in \Cref{subsec:HeuristicEvidence} we discuss what evidence there is for this heuristic.

\subsection{Naive Model}\label{subsec:NaiveModel}
There is already a well accepted model for the statistical behaviour of $\sel{2}(E/\QQ)$, which was first proposed by Poonen--Rains \cite{Poonen_2012}.  The motivation for this is the left and central entries of the diagram: 
\begin{equation}\label{eq:INTROSelmerIntDiagram}
\begin{tikzcd}
    &H^1(\QQ, E[2])\ar{d}{\loc}&\\
    \prod\limits_{v\in\places_\QQ}E(\QQ_v)/2E(\QQ_v)\ar{r}{\prod_v\delta_v}&\prod\limits_{v\in\places_\QQ}H^1(\QQ_v,E[2])&\ar{l}[swap]{\prod_v\delta_{D,v}}\prod\limits_{v\in\places_\QQ}E_D(\QQ_v)/2E_D(\QQ_v).
\end{tikzcd}
\end{equation}

Poonen--Rains observe that $\prod_{v\in\places_\QQ}H^1(\QQ_v,E[2])$ is a quadratic space, and $\im(\loc)$ and $\im\braces{\prod_v\delta_v}$ are maximal isotropic subspaces.  They then predict that since
\[\sel{2}(E/\QQ)=\im (\loc)\cap \im\braces{\prod_v\delta},\]
it is reasonable to model $\sel{2}(E/\QQ)$ as though these two images were in fact random maximal isotropic subspaces of such a quadratic space.  This heuristic agrees with all known results on $2$-Selmer groups: the average size of $\sel{2}(E/\QQ)$ is known to agree with the model by work of Bhargava--Shankar \cite{MR3272925}, and recently the second moment of the size has been shown to agree (assuming a conjectural tail estimate) by Bhargava--Shankar--Swaminathan \cite{swaminathan2021second}.

We note the following lemma.
\begin{lemma}[\cite{MR3043582}*{Lemma 5.2(ii)}]
    For every elliptic curve $E/\QQ$, the image of $\prod_v\delta_{D,v}$ is maximal isotropic with respect to the quadratic form of Poonen--Rains \cite{Poonen_2012}.
\end{lemma}

In particular, $\sel{2}(E_D/\QQ)=\im(\loc)\cap\im\braces{\prod_v\delta_{D,v}}$ is also an intersection of two maximal isotropic subspaces, and so we expect it should have the same distribution as $\sel{2}(E/\QQ)$.  Since the results of Bhargava--Shankar and Bhargava--Shankar--Swaminathan hold for `large families', one can show (see e.g. \cite{Paterson2021}*{Prop 2.9}) that their results give identical evidence towards this twisted application of the Poonen--Rains heuristics.

The na\"ive heuristic is then that
\[\sel{2}(E/\QQ)\cap \sel{2}(E_D/\QQ)=\im(\loc)\cap\im\braces{\prod_v\delta_v}\cap\im\braces{\prod_v\delta_{D,v}},\]
should in fact behave as an intersection of three random maximal isotropic subspaces.  Investigating this model, we discover the following. 
\begin{lemma}\label{lem:PR intersection of 3}
    In the Poonen--Rains model, an intersection of three random maximal isotropic subspaces is trivial with probability $1$.
\end{lemma}
\noindent In particular, the na\"ive model would predict a different result to that in \Cref{thm:INTRO prob nontrivial is positive}.  We postpone a precise statement/proof of this claim to the end of the section, as \Cref{lem:precise intersection of 3 is zero}.

\subsection{Dependence}\label{subsec:Dependence}
Re-examining these subspaces, assuming that $\im\braces{\prod_v\delta_v}$ and $\im\braces{\prod_v\delta_{D,v}}$ are independently random would be a trifle na\"ive:  the images of $\delta_v$ and $\delta_{D,v}$ coincide at all but finitely many $v$, at least those of good reduction for $E$ not dividing $D$ and may also coincide at the remaining places (subject to local conditions on $E$).  If these images often coincide at most of the places of bad reduction of $E$ then there is some clear dependence between the two images.  We formalise this as follows.

\begin{definition}\label{def:local images are dependent}
    Let $\cA\subseteq \Epsilon$ be an infinite subset, and write $\cA(X):=\cA\cap \Epsilon(X)$.  We say that the local image at $D$ is dependent in $\cA$ if 
    \[\limsup_{X\to\infty}\frac{1}{\#\cA(X)}\sum_{E\in\cA(X)}\#\set{p~:~\im(\delta_p)\neq\im(\delta_{D,p})\textnormal{ for $E$ in diagram \Cref{eq:INTROSelmerIntDiagram}}}<\infty.\]
    We say that the local images are independent in $\cA$ for $D$ otherwise.
\end{definition}

\noindent There is a more convenient way to write dependence in terms of the genus theory invariant from \cite{Paterson2021}*{Definition 4.10} (see also \cites{kramer1981arithmetic,MR4400944}), which we will now recall.
\begin{lemma}\label{lem:genus theory gives independence}
    Let $\cA\subseteq \Epsilon$ be an infinite subset, and write $\cA(X):=\cA\cap \Epsilon(X)$.  Let $D\neq 1$ be a squarefree integer, and $K=\QQ(\sqrt{D})$.  Then the local image at $D$ is dependent in $\cA$ if and only if
    \[\cG_\cA^+\braces{K}:=\limsup_{X\to\infty}\frac{1}{\#\cA(X)}\sum_{E\in\cA(X)}g_2\braces{K/\QQ;\,E}<\infty,\]
    where $g_2$ is the genus theory invariant from \cite{Paterson2021} given by
    \[g_2\braces{K/\QQ;\,E}:=\sum_{v\in\places_\QQ}\frac{E(\QQ_v)}{N_{K_w/\QQ_v}E(K_w)}.\]
    The $w$ in each summand above is a choice of place of $K$ lying over $v$.
\end{lemma}
\begin{proof}
    For each place $v\in\places_\QQ$, let $w\in\places_K$ be a place extending $v$.  By \Cref{prop:local corestriction images are intersections of twisted Kummers} there is a short exact sequence of finite abelian groups
    \[0\to \im(\delta_p)\cap\im(\delta_{D,p})\to \frac{E(\QQ_v)}{2E(\QQ_v)}\to \frac{E(\QQ_v)}{N_{K_w/\QQ_v}E(K_w)}\to 0.\]
    Now, if $v$ is a finite place then 
    \[\dim\im(\delta_{D,v})=\dim E(\QQ_v)[2]\#\cO_{K_w}/2\cO_{K_w}=\dim\im(\delta_p).\]
    In particular, the intersection being a strict subset of $E(\QQ_v)/2E(\QQ_v)$ is equivalent to the intersections not being equal.  Similarly the analogous result holds for archimedean places since quadratic twisting does not change the sign of the discriminant.

    We note that by \cite{Paterson2021}*{Lemma 5.3} there exist uniform constants $C_1,C_2>0$ such that for every $E$, $D$ and $v$
    \[C_1g_2(K/\QQ;E)\leq \#\set{v~:~\im(\delta_v)\neq \im(\delta_{D,v})\textnormal{ for $E$ in \Cref{eq:INTROSelmerIntDiagram}}}\leq C_2g_2(K/\QQ;E).\]
    Hence the result holds.
\end{proof}

\subsection{The Heuristic}\label{subsec:MainHeuristic}

If the local images are dependent in an infinite family $\cA$ of elliptic curves, then we do not expect the na\"ive heuristic to hold.  However, if the images at $p\mid D$ are able to differ with some positive probability then we expect when there are a lot of such places then the distribution of $\sel{2}(E/\QQ)\cap\sel{2}(E_D/\QQ)$ should at least be very close to the naive heuristic.  More specifically, we expect that as $\omega(D)\to\infty$ the distribution of the intersection should approximate our na\"ive heuristic.  In light of \Cref{lem:PR intersection of 3}, we make this heuristic precise as follows.

\begin{heur}\label{heur:the heuristic}
    Let $\cA\subseteq \Epsilon$ be a reasonable family of elliptic curves. 
    \begin{enumerate}
        \item\label{heur:INTRO enum no limit} If the local image at a squarefree integer $D\neq 1$ is independent in $\cA$ then the probability that $\sel{2}(E/\QQ)\cap\sel{2}(E_D/\QQ)=0$ is $1$.  That is,
        \[\lim_{X\to\infty}\frac{\#\set{E\in\cA(X)~:~\sel{2}(E/\QQ)\cap\sel{2}(E_D/\QQ)=0}}{\#\cA(X)}=1.\]
        \item\label{heur:INTRO enum limit} If the local images at all but finitely many squarefree $D$ are dependent in $\cA$ and
        \[\liminf_{\omega(D)\to\infty}\cG_\cA^+\braces{\QQ(\sqrt{D})}=\infty,\]
        then the probability that $\sel{2}(E/\QQ)\cap\sel{2}(E_D/\QQ)=0$ tends to $1$ as $\omega(D)\to\infty$.  That is,
        \[\lim_{\omega(D)\to\infty}\lim_{X\to\infty}\frac{\#\set{E\in\cA(X)~:~\sel{2}(E/\QQ)\cap\sel{2}(E_D/\QQ)=0}}{\#\cA(X)}=1\]
    \end{enumerate}
\end{heur}
\begin{rem}
    We cautiously take the reasonable families to be large families (\`a la Bhargava--Shankar \cite{MR3272925}) of elliptic curves (in particular $\Epsilon$), and the family of quadratic twists of a fixed elliptic curve $E/\QQ$.
\end{rem}

\subsection{Evidence}\label{subsec:HeuristicEvidence}
The results in this article, as well as previous work of Morgan and the author \cite{MR4400944} show that \Cref{heur:the heuristic} is correct for certain families.
\subsubsection{Quadratic Twist Families}
In \cite{MR4400944}*{Theorem 6.1+Remark 6.3} it is shown that if $\cA$ is the family of quadratic twists of a fixed elliptic curve with full $2$-torsion then for every squarefree $D\neq 1$
\[\lim_{X\to\infty}\frac{\#\set{E\in\cA(X)~:~\sel{2}(E/\QQ)\cap \sel{2}(E_D/\QQ)=0}}{\#\cA(X)}=1.\]

Moreover, \cite{MR4400944}*{Proposition 5.8} shows that the average size of $g_2(\QQ(\sqrt{D})/\QQ;E)$ for $E\in\cA$ tends to infinity and so by \Cref{lem:genus theory gives independence} the local image at $D$ is independent for all such $D$.  In particular, the heuristic holds in these families.

We note that \cite{MR4400944}*{Proposition 5.8} is more general than this.  In fact, if $\cA'$ is the family of quadratic twists of a fixed base curve $E_0/\QQ$, then so long as $\QQ(\sqrt{D})\not\subseteq\QQ(E_0[2])$ the average size of $g_2(\QQ(\sqrt{D})/\QQ;E)$ for $E\in\cA'$ tends to infinity.  In particular, the local image at all such $D$ is independent and the heuristic predicts that $\sel{2}(E/\QQ)\cap\sel{2}(E_D/\QQ)=0$ for $100\%$ of $E\in\cA'$.  If $E_0$ does not have full $2$-torsion, then the conjecture for this family is open.

\subsubsection{All Elliptic Curves} In the family of all elliptic curves, \Cref{thm:INTRO average size of intersection is product of densities} shows that
\[\lim_{X\to\infty}\frac{\#\set{E\in\cA(X)~:~\sel{2}(E/\QQ)\cap\sel{2}(E_D/\QQ)=0}}{\#\cA(X)}\neq1.\]

Fortunately, \cite{PatMQG}*{Theorem 1.7} shows that the average of $g_2(\QQ(\sqrt{D})/\QQ;E)$ in $\Epsilon$ is approximately $\omega(D)$.  Thus, the local image at $D$ is dependent and $\liminf_{\omega(D)\to\infty}\cG_\cA\braces{\QQ(\sqrt{D})}\to\infty$.  Thus we expect to be in case (2) of \Cref{heur:the heuristic}.

Moreover, \Cref{thm:INTRO prob nonzero goes to zero} shows that
\[\lim_{D\to\infty}\lim_{X\to\infty}\frac{\#\set{E\in\cA(X)~:~\sel{2}(E/\QQ)\cap\sel{2}(E_D/\QQ)=0}}{\#\cA(X)}=1,\]
so our heuristic gives the correct prediction in this case also.

\subsection{Intersections of Isotropic spaces}\label{subsec:heur proofs}
For completeness, we prove the precise form of \Cref{lem:PR intersection of 3}.
\begin{lemma}\label{lem:precise intersection of 3 is zero}
    Let $(V,Q)$ be a $2n$-dimensional weakly metabolic quadratic space over $\FF_p$, with $Q$ taking values in $\FF_p$.  Let $\cI_{V}$ denote the set of maximal isotropic subspaces of $V$.

    \begin{enumerate}
        \item Let $W\in\cI_V$ be a fixed maximal isotropic subspace, and let $Z$ be chosen uniformly at random from $\cI_V$.  Then for each $0\leq r\leq n$, 
        \[\Prob(\dim W\cap Z = r)\leq p^{-\frac{r(r-3)}{2}}\]
        \item Let $A\subseteq V$ be a fixed subspace of dimension $a$, and $Z$ be chosen uniformly at random from $\cI_V$.  Then
        \[\Prob\braces{A\cap Z\neq 0}< p^{a-\frac{(n-1)(n-2)}{2}}\]
        \item Let $W\in \cI_V$ be a fixed maximal isotropic subspace, and $Z_1,Z_2$ be chosen independently and uniformly at random from $\cI_V$.  Then there exists a constant $C$ which is independent of $n$ such that
        \[\Prob(\dim W\cap Z_1\cap Z_2\neq 0)\leq C p^{-(n-1)(n-2)/2}.\]
        In particular, for $r\geq 0$ we have
        \[\lim_{n\to\infty}\Prob(\dim W\cap Z_1\cap Z_2=r)=\begin{cases}
            1&\text{if }r=0,\\
            0&\text{if }r\geq 1.
        \end{cases}\]
    \end{enumerate}
\end{lemma}
\begin{proof}
    We make use of the computations in \cite{Poonen_2012}.
    \begin{enumerate}
        \item This follows from \cite{Poonen_2012}*{Proposition 2.6(e)} since
        \[\Prob\braces{\dim W\cap Z = r}\leq \prod_{j=1}^{r}\frac{p}{p^j-1}\leq p^{r-\frac{r(r-1)}{2}}=p^{-\frac{r(r-3)}{2}}.\]
        \item By \cite{Poonen_2012}*{Proposition 2.6(a),(b)}
        \begin{align*}
            \Prob\braces{A\cap Z\neq 0}
            \leq \sum_{v\in A\backslash\set{0}}\Prob(v\in Z)
            =\frac{p^{a}-1}{\prod_{i=0}^{n-2}(p^{i}+1)}
            <p^{a-\frac{(n-1)(n-2)}{2}}.
        \end{align*}
        \item Using the previous parts,
        \begin{align*}
            \Prob&(W\cap Z_1\cap Z_2\neq 0)\\
            &=\frac{1}{\#\cI_V}\sum_{r=1}^{n}\sum_{\substack{Z_1\in \cI_V\\\dim(Z_1\cap W)=r}}\frac{\#\set{Z_2\in\cI_V~:~Z_2\cap Z_1\cap W\neq 0}}{\#\cI_V}\\
            &<\frac{1}{\#\cI_V}\sum_{r=1}^{n}\sum_{\substack{Z_1\in \cI_V\\\dim(Z_1\cap W)=r}}p^{r-\frac{(n-1)(n-2)}{2}}\\
            &<p^{-\frac{(n-1)(n-2)}{2}}\sum_{r=1}^{n}p^{\frac{-r(r-5)}{2}}
        \end{align*}
        Since the sum above is convergent, the first claim holds. The second claim follows immediately.
    \end{enumerate}
\end{proof}

\bibliography{refs}
\addresseshere
\newpage\appendix

\begin{landscape}
\section{Tate's Algorithm}\label{app:TatesAlg}
Let $F$ be the completion of a number field at a non-archimedean place with residue characteristic $p\geq 5$.  Let $\cO_F$, $v_F$, $\pi_F$ and $k_F$ be the ring of integers, normalised valuation, choice of uniformiser, and residue field.  Let $E:y^2=x^3+Ax+B$ be a minimal integral model for an elliptic curve defined over $F$ (i.e. $v_F(A)\geq 4\implies v_F(B)<6$), and write $P_E(T):=T^3+A\pi_F^{-2}T+B\pi_F^{-3}$.

In \Cref{tab:reduction at primes >=5} we present the well known summary of the outcome of Tate's algorithm (as presented in \cite{silverman1994advanced}) in this setting.

\begin{table}[ht]
\centering
            \begin{tabular}{|c|c|c|c|}
            \hline
            \textbf{Kodaira Type} & \textbf{Subtype} & $\mathbf{c(E/F)}$ & \textbf{condition}\\
            \hline
            $I_0$ 
                &  
                & $1$ 
                & $v_F(4A^3+27B^2)=0$\\

            \hline
            \multirow{3}{*}{$I_n$}
                & split
                    & $n$ 
                    & $v_F(AB)=0$,\ $v_F(4A^3+27B^2)=n$ and $6B\in k_F^{\times2}$\\
                & nonsplit, $n$ even
                    & $2$
                    & $v_F(AB)=0$,\ $v_F(4A^3+27B^2)=n$ and $6B\not\in k_F^{\times2}$\\
                & nonsplit, $n$ odd
                    & $1$
                    & $v_F(AB)=0$,\ $v_F(4A^3+27B^2)=n$ and $6B\not\in k_F^{\times2}$\\
            \hline
            $II$
                & 
                & $1$
                & $v_F(A)\geq 1$ and $v_F(B)=1$\\
            \hline
            $III$
                & 
                & $2$
                & $v_F(A)=1$ and $v_F(B)\geq 2$\\
            \hline
            \multirow{2}{*}{$IV$}
                & split
                    & $3$
                    & $v_F(A)\geq2$, $v_F(B)=2$ and $B\pi_F^{-2}\in k_F^{\times2}$\\
                & nonsplit
                    & $1$
                    & $v_F(A)\geq2$, $v_F(B)=2$ and $B\pi_F^{-2}\not\in k_F^{\times2}$\\
            \hline
            \multirow{3}{*}{$I_0^*$}
                & nonsplit
                    & $1$
                    & $v_F(A)\geq 2$,\ $v_F(B)\geq 3$,\ $v_F(4A^3+27B^2)=6$,\ and $\#\set{\alpha\in k_F~:~P_E(\alpha)=0}=0$\\
                & partially split
                    & $2$
                    & $v_F(A)\geq 2$,\ $v_F(B)\geq 3$,\ $v_F(4A^3+27B^2)=6$,\ and $\#\set{\alpha\in k_F~:~P_E(\alpha)=0}=1$\\
                & completely split
                    & $4$
                    & $v_F(A)\geq 2$,\ $v_F(B)\geq 3$,\ $v_F(4A^3+27B^2)=6$,\ and $\#\set{\alpha\in k_F~:~P_E(\alpha)=0}=3$\\
            \hline
            \multirow{4}{*}{$I_n^*$}
                & split,\ $n$ even
                    & $4$ 
                    & $v_F(A)=2$,\ $v_F(B)=3$,\ $v_F(4A^3+27B^2)=6+n$ and $-(4A^3+27B^2)\pi_F^{-(6+n)}\in k_F^{\times2}$\\
                & nonsplit,\ $n$ even
                    & $2$
                    & $v_F(A)=2$,\ $v_F(B)=3$,\ $v_F(4A^3+27B^2)=6+n$ and $-(4A^3+27B^2)\pi_F^{-(6+n)}\not\in k_F^{\times2}$\\
                & split,\ $n$ odd
                    & $4$
                    & $v_F(A)=2$,\ $v_F(B)=3$,\ $v_F(4A^3+27B^2)=6+n$ and $6B(4A^3+27B^2)\pi_F^{-(9+n)}\in k_F^{\times2}$\\
                & nonsplit,\ $n$ odd
                    & $2$
                    & $v_F(A)=2$,\ $v_F(B)=3$,\ $v_F(4A^3+27B^2)=6+n$ and $6B(4A^3+27B^2)\pi_F^{-(9+n)}\not\in k_F^{\times2}$\\
            \hline
            \multirow{2}{*}{$IV^*$}
                & split
                    & $3$
                    & $v_F(A)\geq3$,\ $v_F(B)= 4$\ and $B\pi_F^{-4}\in k_F^{\times2}$\\
                & nonsplit
                    & $1$
                    & $v_F(A)\geq3$,\ $v_F(B)=4$\ and $B\pi_F^{-4}\not\in k_F^{\times2}$\\
            \hline
            $III^*$
                & 
                & $2$
                & $v_F(A)=3$ and $v_F(B)\geq 5$\\
            \hline
            $II^*$
                & 
                & $1$
                & $v_F(A)\geq 4$ and $v_F(B)=5$\\
            \hline
            \end{tabular}
            \caption{Tate's Algorithm for a minimal model in residue characteristic at least $5$}\label{tab:reduction at primes >=5}
        \end{table}
    \end{landscape}
\end{document}